\documentclass[11pt, twoside, a4paper, english, reqno]{amsart}
\usepackage{amssymb,amsfonts,latexsym,color}
\usepackage{graphics,verbatim}
\usepackage{graphicx}
\newtheorem{theorem}{Theorem}[section]

\newtheorem{lemma}{Lemma}[section]

\newtheorem{remark}{Remark}[section]
\newtheorem{definition}{Definition}[section]

 \newcommand{\<}{\left\langle}
\renewcommand{\>}{\right\rangle}

\newcommand{\eps}{\varepsilon}

\newcommand{\To}{\longrightarrow}

\newcommand{\be} {\begin{equation}}
\newcommand{\ee} {\end{equation}}
\newcommand{\bea} {\begin{eqnarray}}
\newcommand{\eea} {\end{eqnarray}}
\newcommand{\Bea} {\begin{eqnarray*}}
\newcommand{\Eea} {\end{eqnarray*}}
\newcommand{\pa} {\partial}

\newcommand{\al} {\alpha}
\newcommand{\ba} {\beta}
\newcommand{\de} {\delta}
\newcommand{\na}{\nabla}
\newcommand{\ga} {\gamma}

\newcommand{\Om} {\Omega}
\newcommand{\om} {\omega}
\newcommand{\De} {\Delta}
\newcommand{\la} {\lambda}

\newcommand{\nequiv} {\not\equiv}

\newcommand{\no} {\nonumber}
\newcommand{\noi} {\noindent}
\newcommand{\lab} {\label}
\newcommand{\va} {\varphi}

\newcommand{\R}{\mathbb R}
\newcommand{\N}{\mathbb N}
\newcommand{\Rn}{\mathbb R^N}
\newcommand{\Iom}{\int_{\Omega}}
\newcommand{\deb}{\rightharpoonup}
\catcode`\@=11

\makeatletter \@addtoreset{equation}{section} \makeatother

\begin{document}

 \title[Multiplicity results for $(p,\, q)$ fractional Laplace equations]{Multiplicity results for $(p,\, q)$ fractional elliptic equations involving critical nonlinearities}

\author{ Mousomi Bhakta, \,\, Debangana Mukherjee}
\address{ Department of Mathematics, Indian Institute of Science Education and Research, Dr. Homi Bhabha Road,
Pune 411008, India}
\email{mousomi@iiserpune.ac.in, \,\, debangana18@gmail.com }

\date{}

\begin{abstract} In this paper we prove the existence
of infinitely many nontrivial solutions for the class of $(p,\, q)$ fractional elliptic equations involving concave-critical nonlinearities in bounded domains in $\mathbb{R}^N$. Further, when the nonlinearity is of convex-critical type, we establish the multiplicity of nonnegative solutions using variational methods. In particular, we show the existence of at least $cat_{\Omega}(\Omega)$ nonnegative solutions.
\end{abstract}

\maketitle 

\tableofcontents

 \section{ Introduction}
In this article we discuss the existence of multiple nontrivial solutions of $(p,\, q)$ fractional Laplacian equations involving concave-critical type nonlinearities and existence of nonnegative solutions when nonlinearities is of convex-critical type. More precisely, first we consider equations of the type

\begin{equation*}
	(P_{\theta,\la})
	\left\{\begin{aligned}
		(-\De)^{s_1}_p u + (-\De)^{s_2}_q u &=\theta V(x)|u|^{r-2}u +|u|^{p^*_{s_1}-2}u+\la f(x,u),\quad\text{in }\quad \Om, \\
		u &= 0  \quad\text{in }\quad \Rn \setminus \Om,
	\end{aligned}
	\right.
\end{equation*}
where $\Om \subset \Rn$ is a smooth, bounded domain, $\la,\,\theta>0,\, 0<s_2< s_1<1,\, 1<r<q<p<\frac{N}{s_1}$ and $p_s^*=\frac{Np}{N-sp}$ for any $s\in(0,1)$. The functions $f$ and $V$ satisfy certain assumptions, which have been made precise later. Up to a normalisation factor, the non-local Operator $(-\Delta)^s_a$ ($a\geq 1$) , is defined as follows:
\begin{align} \label{frac s_a}
	(-\Delta)^s_a u(x)=\lim_{\eps\to 0}\int_{\mathbb{R}^N\setminus B_\eps(x)}\frac{|u(y)-u(x)|^{a-2}(u(y)-u(x))}{|x-y|^{N+as}}dy,\,\,\,x\in\mathbb{R}^N. 
\end{align}

For $s_1=s_2=1$, the equation in $(P_{\theta, \lambda})$ reduces to the $(p,\, q)$ Laplacian problem which appears in more general reaction-diffusion system 
\be\lab{7-1-1}u_t=\text{div}\big(a(u)\na u\big)+g(x,u),\ee
where $a(u)=|\na u|^{p-2}\na u+|\na u|^{q-2}\na u$. This system has a wide range of applications in Physics which include biophysics, plasma physics and chemical reaction-diffusion system, etc. In such applications, the function $u$ describes a concentration, the first term on the right-hand side of \eqref{7-1-1} corresponds to the diffusion with a diffusion coefficient $a(u)$ and the second
one is the reaction and relates to sources and loss processes. Typically, in chemical and biological applications, the reaction
term $g(x,u)$ has a polynomial form with respect to the concentration $u$. Consequently, quasilinear elliptic boundary value problems involving this operator have been widely studied in the literature (see e.g., \cite{BMV, MP1, MP} and the references there-in). In particular, proving the existence and  multiplicity of nontrivial solutions and nonnegative solutions were of major interest in many articles, see \cite{CMP, LZ, YinYang1, YinYang2} and the references there-in.

We also observe that the functional associate to the operator in $(P_{\theta, \lambda})$ is originally connected to Homogenization theory \cite{Zhi-2}. For example in the local case $s_1=s_2=1$, the functional associated to the operator in $(P_{\theta, \lambda})$ falls in the realm of general model functional
$$\mathcal{P}(w,\Om):= \displaystyle\int_{\Om} \big(|\na w|^p+a(x)|\na w|^q\big)dx,$$ where $a(.)\geq 0$, was extensively studied in \cite{GM-2, GM-1, GM-3}. This kind of functional was first introduced by Zhikov \cite{Zhi-1, Zhi-2} in order to produce models for strongly anisotropic materials. They intervene in Homogenisation theory and Elasticity, where the coefficient $a(.)$ for instance dictates the geometry of a composite made by two different materials. 

 When $p=q$ and $s_1=s_2$, $(P_{\theta,\la})$ reduces to $p$-fractional type equations with concave-convex nonlinearities. In recent years, existence and multiplicity result for nontrivial, positive and sign-changing solutions for the $p$-fractional type equations with concave-convex nonlinearities have gained considerable interest. In this regard we cite some of the related recent works \cite{BM, CD, CS, GS} (also see the references there-in).

In the nonlocal case $s\in(0,1)$ and $p=q=2$, equations with two nonlocal operators have also started gaining interest in the past few years starting with the work of Chen, Kim, Song, et al, see \cite{CKS}, \cite{CKSV}. Very recently in \cite{A} and \cite{CB}, authors studied some existence and multiplicity results for $(p,q)$ fractional Laplacian type equations in $\Rn$. But as of our knowledge, there is no article so far where the nonlinear analysis involving $(p, q)$ fractional Laplacian operator or combination of two linear fractional Laplacian operator have been carried out in bounded domain in $\Rn$ in the spirit of above mentioned works. The aforementioned results are motivation for this present paper, where we study the existence and multiplicity results for the equations involving $(p,q)$ fractional Laplacian operator with concave-critical or convex-critical nonlinearities. 
In this regards we also like to mention that very recently, in \cite{FLP} existence of non-negative solutions for system of equations involving fractional $(p,q)$ Laplace operator have been studied. 

\vspace{3mm}

 For $ p\geq 1$ and $s\in(0,1)$, we denote the standard fractional Sobolev space by $W^{s,p}(\Omega)$ endowed with the norm
$$
\|{u}\|_{W^{s,p}(\Om)}:=\|{u}\|_{L^p(\Om)}+\left(\int_{\Om\times\Om} \frac{|u(x)-u(y)|^p}{|x-y|^{N+sp}}dxdy\right)^{1/p}.
$$
We set $Q:=\R^{2N}\setminus (\Om^c \times \Om^c)$, where $\Om^c=\Rn \setminus \Om$ and define $$
X_{s,p}(\Om):=\Big\{u:\mathbb{R}^N\to\mathbb{R}\mbox{ measurable }\Big|u|_{\Omega}\in L^p(\Omega)\mbox{ and }
\int_{Q} \frac{|u(x)-u(y)|^p}{|x-y|^{N+sp}}dxdy<\infty\Big\}.
$$
The space $X_{s,p}(\Om)$ is endowed with the norm defined as
$$\|u\|_{s,p}:=|u|_p+\left(\int_{Q} \frac{|u(x)-u(y)|^p}{|x-y|^{N+sp}}dxdy\right)^{1/p},$$
where $|u|_p=\|u\|_{L^p(\Om)}$. Note that in general $W^{s,p}(\Om)$ is not same as $X_{s,p}(\Om)$ as $\Om\times\Om$ is strictly contained in $Q$.

Next, we define $X_{0,s,p}(\Om) :=\Big\{u \in X_{s,p} : u=0 \quad\text{a.e. in}\quad \Rn \setminus \Om\Big\} $ or equivalently 
as $\overline{C_0^\infty(\Om)}^{X_{s,p}(\Om)}$. It is well-known that for $p>1$,  $X_{0,s,p}(\Om)$ is a uniformly convex Banach space endowed with the norm   
$$\|u\|_{0,s,p}=\left(\int_{Q} \frac{|u(x)-u(y)|^p}{|x-y|^{N+sp}}dxdy\right)^{1/p}.$$	
	
Since $u=0$ in $\Rn\setminus\Om,$ the above integral can be extended to all of $\mathbb{R}^N.$ The embedding
 $X_{0,s,p}(\Om)\hookrightarrow L^r(\Om)$ is continuous for any $r\in[1,p^*_s]$ and compact for $r\in[1,p^*_s).$  Moreover, for $1<q\leq p$, $X_{0,s_1,p}(\Om)\subset X_{0,s_2,q}(\Om)$ (see Lemma \ref{Lemma4} in Section 2). For further details on $X_{0,s,p}(\Om)$ and it's properties we refer \cite{NePaVal}. 
 
 \vspace{2mm}
 	
Throughout this article we assume the functions $V(\cdot),\, f(\cdot,\cdot)$ satisfy the following:
\begin{itemize}
	\item [\bf(A1)]
	$V \in L^{\infty}(\Om)$ and there exists $\sigma>0,\,\eta>0$ such that $V(x)>\sigma>0$ for all $x \in \Omega$ and 
	$$\Iom V(x)|u|^r\,dx \leq \eta \|u\|_{0,s_2,r}^r$$ for all $u \in X_{0,s_2,r}(\Om)$.
	\item [\bf(A2)]
	$|f(x,t)| \leq a_1|t|^{\al-1}+a_2|t|^{\ba-1}$ for all $x \in \Om, \, t \in \R, \,a_1,\, a_2>0,\,1<\al,\,\ba<p^*_{s_1}.$
	\item [\bf(A3)]
	There exists $a_3>0$ and $l \in (1,p)$ such that 
	$$f(x,t)t-p^*_{s_1}F(x,t) \geq -a_3|t|^{l}$$ for all $x \in \Om, t \in \R$ where
	$F(x,t)=\int_0^t f(x,\tau)d\tau.$ 
	\item [\bf(A4)]
	$f(x,t)>0$ for all $x \in \Om,t \in \R^+$ and
	$f(x,t)=-f(x,-t)$ for all $x \in \Omega, t \in \R.$
\end{itemize}

 \begin{definition}\label{sol} We say that $u\in X_{0,s_1,p}(\Om)$ is a weak solution of $(P_{\theta,\la})$ if for all $\phi \in X_{0,s_1,p}(\Om)$, we have
 	\Bea
 		&&\int_{\R^{2N}}\frac{|u(x)-u(y)|^{p-2}(u(x)-u(y))(\phi(x)-\phi(y))}{|x-y|^{N+ps_1}}dxdy \\
 		&&+ \int_{\R^{2N}}\frac{|u(x)-u(y)|^{q-2}(u(x)-u(y))(\phi(x)-\phi(y))}{|x-y|^{N+qs_2}}dxdy \\
 		&&= \theta\Iom  V(x) |u(x)|^{r-2}u(x)\phi(x)dx
 		+ \Iom |u(x)|^{p^*_{s_1}-2}u(x)\phi(x)dx+ \la\Iom f(x,u)\phi dx.
 	\Eea
 \end{definition}

\noi Our first main result is the following:
\begin{theorem}\label{Thm}
Let $0<s_2< s_1<1$,  $1<r<q<p<\frac{N}{s_1}$ and assumptions (A1)-(A4) being satisfied. Then there exists $\la^*>0$ such that for any $\la \in (0,\la^*),$ there exists $\theta^*>0$ such that for any $\theta \in (0,\theta^*),$ problem $(P_{\theta,\la})$ has infinitely many  nontrivial weak solutions in $X_{0,s_1,p}(\Om).$
\end{theorem}

Our next goal is to study the nonnegative solutions to $(P_{\theta,\la})$.  For  $V(x)\equiv1$ and $\lambda=0,$ we investigate the nonnegative solutions of  $(P_{\theta,\la})$ and prove the following results:

\begin{theorem}\label{thm1}
Let $0<s_2<s_1<1$ and $2\leq q<p<r<p^*_{s_1}.$ Then there exists $\theta^*>0$ such that for any $\theta>\theta^*,$ the  problem
  \begin{equation}
 (P)\label{P}
 \left\{\begin{aligned}
 (-\De)^{s_1}_p u + (-\De)^{s_2}_q u &=\theta|u|^{r-2}u +|u|^{p^*_{s_1}-2}u  \quad\text{in }\quad \Om, \\
 u &>0 \quad\text{in}\quad \Om, \\
 u&=0 \quad\text{in}\quad \Rn \setminus \Omega.
 \end{aligned}
 \right.
 \end{equation}	
 has a   nontrivial nonnegative weak solution. 
\end{theorem}

To state our next theorem, we need the following definition.

\begin{definition}
Let $M$ be a topological space and consider a closed subset $A\subset M$.
We say that $A$ has category $k$ relative to $M$ ($cat_M(A)=k$), if $A$
is covered by $k$ closed sets $A_j, \, 1\leq j\leq k$, which are contractible in $M$, and if $k$ is minimal with this property. If no such finite covering exists, we define
$cat_M(A)=\infty$. Moreover, we define $cat_M(\emptyset)=0$.
\end{definition}

Using Lusternik--Schnirelmann category theory, we prove our next result.

\begin{theorem}\label{thm2} 
Let $0<s_2<s_1<1$ and
$$N>p^2s_1,\quad 2\leq q<\frac{N(p-1)}{N-s_1}<p \leq \max \{p,\, p^*_{s_1}-\frac{q}{q-1}\}<r<p^*_{s_1}.$$
Then there exists $\theta_{**}>0$ such that for any $\theta \in (0,\theta_{**}),$ problem
	(P) has at least $\text{cat}_\Omega(\Omega)$ nontrivial nonnegative solutions in $X_{0,s_1,p}(\Om).$ 
\end{theorem}

The rest of the paper is organized as follows. Section 2 is preliminaries, where we prove $X_{0,s_1,p}(\Om)$ is embedded in  $X_{0,s_2,q}(\Om)$ and the concentration compactness lemma for the p-fractional case. Section 3, Section 4 and Section 5 deals with the proof of Theorem \ref{Thm}, Theorem \ref{thm1} and Theorem \ref{thm2} respectively. The paper is concluded with an appendix where we recall the statement of classical deformation lemma, general mountain pass lemma and some standard properties of genus. 

Before concluding the introduction, we would like to remark that as a future direction it would be interesting to study similar kind of problems by considering even more general kernel, again both fractional and nonlinear, but taking into account rough coefficients. Thus, by replacing the fractional Laplacians in ${(P_{\theta,\la})}$ with the nonlinear integro-differential operator given by
$$\mathcal{L}_K u(x)=P.V. \int_{\Rn}|u(y)-u(x)|^{p-2}(u(y)-u(x))K(x,y)dy,\, x \in \Rn$$
where the symmetric function $K$ is a Gagliardo-type kernel (namely, an $(s, p)$-kernel) with measurable coefficients. In this respect, the strictly related papers by Di Castro et Al. (\cite{Ref-1},\cite{Ref-2}), which deal with the general nonlinear fractional operators above could be the starting point in order to study such a generalization.

\section{Preliminary}
\subsection{Besov-Sobolev embeddings}
In this subsection first we define Besov space of $\Rn$ and $\Om$. For $1\leq i \leq N$ and $h\in\R$, let $\De_i^h u$ denote the difference quotient defined by $\De_i^h u(x)=u(x+he_i)-u(x), \quad x \in \Rn.$

\begin{definition}\cite[pg. 415]{Leoni}
 Let $1 \leq p,\, q \leq \infty$ and $0<s<1$.  A function $u \in L^1_{loc}(\Rn)$ belong to the Besov space $B^s_{p,q}(\Rn)$ if
$$\|u\|_{B^s_{p,q}(\Rn)}=|u|_{L^p(\Rn)}+[u]_{B^s_{p,q}(\Rn)} <\infty,$$ 
where 
\begin{equation}
[u]_{B^s_{p,q}(\Rn)}=\begin{cases}
\displaystyle\sum_{i=1}^{N}\bigg(\displaystyle \int_{0}^{\infty}\|\De_i^h u\|^q_{L^p(\Rn)}\frac{dh}{h^{1+sq}}\bigg)^{\frac{1}{q}},\,\ q<\infty,\\
\displaystyle\sum_{i=1}^{N}\sup_{h> 0}\frac{1}{h^s}\|\De^h_i u\|_{L^p(\Rn)},\,\ q=\infty.
\end{cases}
\end{equation}	
\end{definition}

\begin{definition} Let $\mathcal{D'}(\Om)$ denote the set of all distributions over $\Om$. For $1\leq p,\, q \leq \infty$, and $0<s<1$,  we set
$$B^s_{p,q}(\Om)= \{u \in \mathcal{D'}(\Om) : \,\exists\, g \in B^s_{p,q}(\Rn) \quad\mbox{with}\quad g|_{\Om}=u \}$$
and $$\|u\|_{B^s_{p,q}(\Om)}=\inf_{g \in B^s_{p,q}(\Rn),\, g|_{\Om}=u} \|g\|_{B^s_{p,q}(\Rn)}.$$	
$B^s_{p,q}(\Om)$ is called the Besov Space over $\Om$.
\end{definition}	

For more details about Besov space, we refer \cite{Leoni} and \cite{Tr}.

\begin{lemma}\lab{embed}
	Let $1 \leq q\leq p \leq \infty$ and $0<s_2< s_1<1.$ Then 
	$$W^{s_1, p}(\Om)\subset W^{s_2, q}(\Om). $$	\end{lemma}
\begin{proof}
Since $q\leq p$ and $s_2< s_1$ implies $s_2-\frac{N}{q}<s_1-\frac{N}{p}$, from \cite[Theorem (i), pg. 196]{Tr}, we have 
$$B^{s_1}_{p,p}(\Om) \subset B^{s_2}_{q,q}(\Om).$$ 
Further, from \cite[pg. 209]{Tr}), it follows that $|u|_{L^p(\Om)}+\displaystyle\bigg(\int_{\Om\times\Om}\frac{|u(x)-u(y)|^p}{|x-y|^{N+sp}}dxdy \bigg)^{\frac{1}{p}}$
is an equivalent norm for $\|u\|_{B^s_{p,p}(\Om)}$. Therefore, $B^s_{p,p}(\Om)=W^{s,p}(\Om)$ for $1 \leq p \leq \infty$ and $0<s<1$. Hence the lemma follows.
\end{proof}	

Note that the assertion of the above Lemma fails when $s_1=s_2$, see \cite{MirS} for the counterexample.  

\begin{lemma}\label{Lemma4}
Let $0<s_2< s_1<1,\, 1<q \leq p$ and $\Om$ be a smooth bounded domain in $\Rn$, where $N>s_1p$. Then $X_{0,s_1,p}(\Om)\subset X_{0,s_2,q}(\Om)$ and there exists $C=C(|\Om|,N,p,q,s_1, s_2)>0$ such that $$\|u\|_{0,s_2,q} \leq C\|u\|_{0,s_1, p} \quad \forall \, u \in X_{0,s_1,p}(\Om).$$
\end{lemma}
\begin{proof}

Let $u \in X_{0,s_1,p}(\Om)$. Then $u \in W^{s_1,p}(\Rn)$ with $u \equiv 0$ a.e. in $\Rn \setminus \Om.$  Note that, thanks to H\"{o}lder inequality and Sobolev inequality, we have
\Bea
\|u\|^p_{W^{s_1,p}(\Om)}&=&|u|^p_{p}+\int_{\Om \times \Om}\frac{|u(x)-u(y)|^p}{|x-y|^{N+s_1p}}dxdy\\
&\leq& |u|_{p^*_{s_1}}^p|\Om|^{1-\frac{p}{p^*_{s_1}}}+\int_{\Rn \times \Rn} \frac{|u(x)-u(y)|^p}{|x-y|^{N+s_1p}}dxdy\\
&\leq& (C|\Om|^{1-\frac{p}{p^*_{s_1}}}+1)\|u\|_{0,s_1,p}^p.\\
\Eea
This proves that $X_{0,s_1,p}(\Om)\subset W^{s_1,p}(\Om)$. Consequently, by Lemma \ref{embed}  we also have $W^{s_1,p}(\Om) \subset W^{s_2,q}(\Om).$  As a result, $u \in W^{s_2,q}(\Om)$ with $u \equiv 0$ a.e. in $\Rn \setminus \Om.$ Further, as $\pa \Om$ is smooth,  
the embedding $W^{s_2,q}(\Om) \hookrightarrow W^{s_2,q}(\Rn)$ is continuous, that is,
\begin{equation}\label{n-1}
\|u\|_{W^{s_2,q}(\Rn)} \leq C(|\Om|,q,N)\|u\|_{W^{s_2,q}(\Om)} \quad\mbox{for all}\quad u \in W^{s_2,q}(\Om).
\end{equation}
Therefore, 
\begin{equation}\label{n-2}
\|u\|_{W^{s_2,q}(\Rn)} \leq C(|\Om|,N,p,q,s_1, s_2)\|u\|_{0,s_1,p} \quad\mbox{for all}\quad u \in X_{0,s_1,p}(\Om).
\end{equation}	
Since, $\|u\|_{0,s_2,q}^q \leq \|u\|_{W^{s_2,q}(\Rn)}^q$,	it follows
\begin{equation}\label{n-3}
\|u\|_{X_{0,s_2,q}(\Om)} \leq C(|\Om|,N,s_1,s_2,p,q)\|u\|_{X_{0,s_1,p}(\Om)} \quad\mbox{for all}\quad u \in X_{0,s_1,p}(\Om).
\end{equation}
Hence the lemma follows.
\end{proof}

\subsection {Concentration-compactness}
For $s\in(0,1)$, define
 $$\dot{W}^{s,p}(\Rn):=\bigg\{u\in L^{p^*_s}(\Rn) :  \displaystyle\int_{\R^{2N}} \frac{|u(x)-u(y)|^p}{|x-y|^{N+sp}}dxdy<\infty\bigg\}$$ and 
\be\lab{S}S_{s,p}=\inf_{u\in\dot{W}^{s,p}(\Rn)\setminus\{0\}} \frac{\displaystyle\int_{\R^{2N}} \frac{|u(x)-u(y)|^p}{|x-y|^{N+sp}}dxdy}{\bigg(\displaystyle\int_{\Rn}|u|^{p^*_s}\bigg)^\frac{p}{p^*_s}}.\ee

Next, we fix some notations: 
$D^su(x):=\displaystyle\int_{\Rn}\frac{|u(x)-u(y)|^p}{|x-y|^{N+sp}}dy$. Thus, $|D^su|_p^p=\|u\|_{0,s,p}^p$, $C_c(\Rn)$ denotes the set of all continuous functions with compact support. 
$\|\mu\|:=\int_{\Rn}d\mu$. $\mathcal{M}(\Rn)$ denotes the space of finite measures on $\Rn.$ We say a sequence $(\mu_n)$ converges weakly to $\mu$ in $\mathcal{M}(\Rn),$ if $$\<\mu_n,\phi\>:=\int_{\Rn}\phi\,d\mu_n \to \<\mu,\phi\> \quad\forall\quad \phi \in C_c(\Rn)$$
 and it is denoted by $\mu_n \deb \mu$.

\begin{theorem}\label{Willem-lemma}
 Let $s\in (0,1)$	 and $p>1$. Assume $\{u_n\} \subset X_{0,s,p}(\Om)$ is a nonnegative sequence such that $|u_n|_{p^*_s}=1$ and $\|u_n\|_{0,s,p}^p \to S_{s,p}$ as $n \to \infty.$ Then, there exists a sequence $\{ y_n,\la_n\} \in \Rn \times \R^+$ such that
 		\begin{equation}\label{v-n}
 			v_n(x):=\la_n^{\frac{(N-sp)}{p}}u_n(\la_nx+y_n)
 		\end{equation} 
 		has a convergent subsequence (still denoted by $v_n$) such that $v_n \to v$ in $\dot{W}^{s,p}(\Rn)$ where $v(x)>0$ in $\Rn.$  In particular, there exists a minimizer for $S_{s,p}.$ Moreover, we have, $\la_n \to 0$ and $y_n \to y \in \overline{\Om}$ as $n \to \infty.$
 	\end{theorem}
 
\begin{proof}
For $p=2$, this lemma has been proved by Palatucci-Pisante in \cite[Theorem 1.3]{PP}. For general $p>1$, using Lemma \ref{Willem-lemma 1.40} (see the next lemma), the proof can be completed following the similar steps  as in \cite[Lemma 1.41]{W}(also see \cite[Section I.4, Example (iii)]{Lions} ). We omit the details.
\end{proof}		

\begin{remark} that the result proven by Palatucci and Pisante is \cite[Theorem 1.3]{PP} is still valid even for $s > 1$ and the subsequent proof is much in the spirit of  \cite{PP-1}.
\end{remark}

\begin{lemma}\label{Willem-lemma 1.40}
 Let $s\in (0,1)$	 and $p> 1$. Assume $\{u_n\}$ be a sequence in $\dot{W}^{s,p}(\Rn)$ such that 
	\begin{equation}\label{C-1}
	\begin{cases}
	u_n \deb u \quad\mbox{in}\quad \dot{W}^{s,p}(\Rn),\\
	|D^s(u_n-u)|^p \deb \mu \quad\mbox{in}\quad \mathcal{M}(\Rn),\\
	|u_n-u|^{p^*_s} \deb \nu \quad\mbox{in}\quad \mathcal{M}(\Rn),\\
	u_n \to u \quad\mbox{ a.e. on}\quad \Rn,
	\end{cases}
	\end{equation}
	and define
	\begin{equation}\label{C-2}
	\begin{cases}
	\mu_\infty:=\lim_{R \to \infty} \limsup_{n \to \infty} \displaystyle\int_{|x|\geq R}|D^s u_n|^pdx,\\
	\nu_\infty:=\lim_{R \to \infty} \limsup_{n \to \infty}  \displaystyle\int_{|x|\geq R}|u_n|^{p^*_s}dx.
	\end{cases}
	\end{equation}

Then, we have
\begin{equation}\label{C-3}
S_{s,p}\|\nu\|^{\frac{p}{p^*_s}} \leq \|\mu\|,
\end{equation}
\begin{equation}\label{C-4}
S_{s,p}\nu_\infty^{\frac{p}{p^*_s}} \leq \mu_\infty,
\end{equation}
\begin{equation}\label{C-5}
\limsup_{n \to \infty} |D^s u_n|^p_p=|D^s u|_p^p+\|\mu\|+\mu_\infty,
\end{equation}
\begin{equation}\label{C-6}
\limsup_{n \to \infty} |u_n|^{p^*_s}_{p^*_s}=|u|_{p^*_s}^{p^*_s}+\|\nu\|+\nu_\infty.
\end{equation}
Moreover, if $u=0$ and $S_{s,p}\|\nu\|^{p/{p^*_s}}=\|\mu\|,$ then $\mu, \nu$ are concentrated at a single point. 
\end{lemma}

\begin{remark}(i) In the local case, Lemma \ref{Willem-lemma 1.40} has been proved in \cite[Lemma I.1]{Lions} (see also \cite[Lemma 1.40]{W} for $s=1,\, p=2$). For the concentration-compactness result
in the bounded domain, i.e.,  when $u_n\deb u$ in ${W}^{s,p}_0(\Om)$, we cite 
\cite[Theorem 2.5]{MS}. Combining the ideas of \cite{Lions}, \cite{MS} and \cite{W}, one expects the above lemma to hold for general $s\in(0,1)$ and $p\geq 1$ (see \cite[Section I.4]{Lions}), but as best of our knowledge this lemma has not been proved exclusively anywhere. For $s\in(0,1),\, p=2$, concentration-compactness result in $\Rn$ has been proved in \cite{DMV} using the harmonic extension method of Caffarelli-Silvestre, which clearly does not work for $p\not=2$ case. Therefore we give here the proof for reader's convenience. Our proof is much different from \cite{MS} and \cite{A},  in both the cases the tightness of the sequence was assumed, where as we have not taken any such assumption.

(ii) It's easy to see that for $\phi\in C^{\infty}_0(\Rn)$, $D^s\phi$ does not have compact support. Thus, when $u_n\deb 0$ in $\dot{W}^{s,p}(\Rn)$, one can not just apply Rellich compactness result to 

 $\lim_{n\to\infty}\int_{\Rn}|u_n|^p|D^s\phi|^p$ in order to pass the limit. This makes the situation much different from the local case \cite{Lions} or the nonlocal case when $u_n\deb u$ in $W^{s,p}_0(\Om)$, which was treated in \cite{MS}.  
\end{remark}

 \begin{proof}
 Let us first consider the case $u \equiv 0.$
 
 \vspace{2mm}
 
{\bf Step 1}: In this step we prove $S_{s,p}(\|\nu\|)^{p/{p^*_s}} \leq \|\mu\|.$

\vspace{2mm}

Choosing $\phi \in C_0^{\infty}(\Rn)$ and applying Sobolev inequality, we have
\bea\label{C-7}
S_{s,p}|u_n \phi|_{p^*_s}^{p^*_s} \leq \|u_n \phi\|_{0,s,p}^p&=&|D^s(u_n \phi)|_p^p\no\\
&\leq& (1+\theta)\int_{\Rn}|D^s u_n|^p|\phi|^pdx+c_\theta \int_{\Rn}|D^s \phi|^p|u_n|^pdx,
\eea 	
where, in the last line we have used  \cite[(2.1)]{MS}. Let, supp$(\phi) \in B(0,r)$ for some $r>0.$ Then for a.e. $|x|>r$, 
	\be\label{V}
	|D^s \phi|^p(x)= \int_{B(0,r)}\frac{|\phi(y)|^p}{|x-y|^{N+sp}}dy \leq\int_{B(0,r)}\frac{|\phi(y)|^p}{{(|x|-r)}^{(N+sp)}}dy
	\leq \frac{|\phi|_p^p}{(|x|-r)^{N+sp}}, 
	\ee
	Fix, $R_\theta>r$ large enough (will be chosen later) . Then,
	\bea\lab{11-1-2}c_\theta\int_{\Rn}|D^s \phi|^p|u_n|^pdx&=&c_\theta\int_{B(0,R_\theta)}|D^s \phi|^p|u_n|^pdx+c_\theta\int_{\Rn \setminus B(0,R_\theta)}|D^s \phi|^p|u_n|^pdx\no\\
&=:&J_1(n)+J_2(n),\eea 
We observe that as $u_n \deb u$ in $\dot{W}^{s,p}(\Rn)$ and $u \equiv 0$, it holds 
$u_n \to 0$ in $L^p_{loc}(\Rn).$ Also, $\phi \in C_0^\infty(\Rn)$ implies, $|D^s\phi|^p \in L^\infty(\Rn).$ Therefore,
\be\lab{11-1-1}
\lim_{n\to\infty} J_1(n)=0.
\ee
Clearly, \begin{equation}\label{IV}
|u_n|_{p^*_s}^{p^*_s} \leq c_1 \quad\mbox{for all}\quad n \geq 1
\end{equation}
	for some $c_1>0.$
Consequently, applying H\"older inequality followed  by (\ref{V}) yields
\bea\label{VI}
J_2(n) &\leq& c_\theta c_1^{p/p^*_s}\displaystyle\left(\int_{\Rn\setminus B(0, R_{\theta})}|D^s \phi|^\frac{N}{s}dx\right)^\frac{sp}{N}\no\\
&\leq& c_\theta c_1^{p/p^*_s}|\phi|_p\bigg(\om_N \int_{R_\theta}^{\infty}\frac{t^{N-1}}{(t-r)^{(N+sp)\frac{N}{sp}}}dt \bigg)^{\frac{sp}{N}},
\eea
where $\om_N$ denotes the surface measure of unit sphere in $R^N$. A straight-forward computation yields,
	\begin{equation}
	\int_{R_\theta}^{\infty}\frac{t^{N-1}}{(t-r)^{(N+sp)\frac{N}{sp}}}dt=2^{N-2}\bigg[\frac{sp}{N^2}\frac{1}{(R_\theta-r)^{N^2/sp}}+\bigg(\frac{r^{N-1}sp}{N(N+sp)-sp}\bigg)\frac{1}{(R_\theta-r)^{\frac{N(N+sp)}{sp}-1}}  \bigg]^{\frac{sp}{N}}.
	\end{equation}	
	Choose $R_{\theta}$ such that
	\begin{equation}\label{VII}
	c_\theta c_1^{p/p^*_s}\om_N^{sp/N}|\phi|_p2^{N-2}\bigg[\frac{sp}{N^2}\frac{1}{(R_\theta-r)^{N^2/sp}}+\bigg(\frac{r^{N-1}sp}{N(N+sp)-sp}\bigg)\frac{1}{(R_\theta-r)^{\frac{N(N+sp)}{sp}-1}}  \bigg]^{\frac{sp}{N}}< \theta.
	\end{equation}
As a consequence, 
\be\lab{11-1-4}
J_2(n)<\theta, \quad\forall\,\, n\geq 1.
\ee	
Combining this with \eqref{11-1-1} and \eqref{11-1-2} yields
$$\lim_{n\to\infty} c_\theta\int_{\Rn}|D^s \phi|^p|u_n|^pdx <\theta.$$
Hence, taking the limit $n\to\infty$ in \eqref{C-7} we obtain
\begin{equation}\label{C-9}
S_{s,p}\bigg(\int_{\Rn}|\phi|^{p^*_s}d\nu\bigg)^{p/{p^*_s}} \leq (1+\theta) \int_{\Rn}|\phi|^p d\mu+ \theta.\end{equation}
Since $\theta>0$ is arbitrary, so letting $\theta \to 0$ in (\ref{C-9}) gives
\begin{equation}\label{C-10}
S_{s,p}\bigg(\int_{\Rn}|\phi|^{p^*_s}d\nu\bigg)^{p/p^*_s} \leq \int_{\Rn}|\phi|^pd\mu \quad\forall\quad \phi \in C_0^{\infty}(\Rn).
\end{equation}
Hence, taking supremum over $C_0^{\infty}(\Rn)$, we get
$$S_{s,p}(\|\nu\|)^{p/{p^*_s}} \leq \|\mu\|.$$ 

\vspace{2mm}

{\bf Step 2}: In this step we prove $S_{s,p} \nu_\infty^{p/{p^*_s}} \leq \mu_\infty.$

\vspace{2mm}

For this first fix $R>1$  and choose $\psi_R \in C^{\infty}(\Rn)$ be such that
\begin{equation}
\psi_R(x)=\begin{cases}
1, \,\ |x|>R+1,\\
0, \,\ |x|<R,\\
\end{cases}
0 \leq \psi_R \leq 1 \quad\mbox{in}\quad \Rn.
\end{equation}
 Thanks to Sobolev inequality, we have
$$S_{s,p}\bigg(\int_{\Rn}|\psi_R u_n|^{p^*_s}dx \bigg)^{p/{p^*_s}} \leq \int_{\Rn}|D^s(u_n \psi_R)|^pdx.$$
Therefore, as before we get
\be\lab{12-1-2}S_{s,p}\bigg(\int_{\Rn}|\psi_R|^{p^*_s} |u_n|^{p^*_s}dx \bigg)^{p/{p^*_s}} \leq (1+\theta) \int_{\Rn}|D^su_n|^p |\psi_R|^pdx+c_\theta \int_{\Rn}|u_n|^p|D^s \psi_R|^p dx.\ee
Doing an easy computation, it follows that $D^s \psi_R \in L^\infty(\Rn).$ Therefore, for any $\tilde R>R+1$,
\bea\lab{12-1-1}  c_\theta \limsup_{n \to \infty} \int_{\Rn}|u_n|^p|D^s \psi_R|^p dx&=& c_\theta \limsup_{n \to \infty} \int_{B(0,\tilde R)}|u_n|^p|D^s \psi_R|^p dx\no\\
&&+ c_\theta\limsup_{n \to \infty} \int_{\Rn\setminus B(0,\tilde R)}|u_n|^p|D^s \psi_R|^p dx\no\\
&=&c_\theta\limsup_{n \to \infty} \int_{\Rn\setminus B(0,\tilde R)}|u_n|^p|D^s \psi_R|^p dx.
\eea
Moreover, for $x\in \overline{B(0, R+1)}^c$,
\Bea
|D^s\psi_R(x)|=\int_{\Rn}\frac{|1-\psi_R(y)|^p}{|x-y|^{N+sp}}dy
&\leq&2^{p-1}\int_{B(0,R+1)}\frac{1+\psi_R(y)^p}{|x-y|^{N+sp}}dy\\
&\leq&\frac{2^{p-1}}{(|x|-(R+1))^{N+sp}}\int_{B(0,R+1)}(1+\psi_R(y)^p)dy\\
&\leq&\frac{2^{p-1}\al_N}{(|x|-(R+1))^{N+sp}}\big(2(R+1)^N-R^N\big) ,
\Eea
where $\al_N$ is volume of unit ball in $\Rn$. Therefore, doing the similar analysis as in Step 1, we get an existence of $\tilde R>R+1$, for which 
$$c_\theta \int_{\Rn\setminus B(0,\tilde R)}|u_n|^p|D^s \psi_R|^p dx<\theta.$$ Hence, combining this along with \eqref{12-1-1} and \eqref{12-1-2} and then taking $\theta\to 0$ yields
\begin{equation}\label{C-11}
\limsup_{n \to \infty} S_{s,p} \bigg(\int_{\Rn}|\psi_R|^{p^*_s} |u_n|^{p^*_s}dx \bigg)^{p/{p^*_s}} \leq  \limsup_{n \to \infty} \int_{\Rn}|\psi_R|^p|D^s u_n|^pdx.
\end{equation}

On the other hand, we have
$$\int_{|x|>R+1}|D^s u_n|^pdx \leq \int_{\Rn}\psi_R^p |D^s u_n|^pdx \leq \int_{|x| \geq R} |D^s u_n|^pdx$$ and
$$\int_{|x|>R+1}|u_n|^{p^*_s}dx \leq \int_{\Rn} |u_n|^{p^*_s}\psi_R^{p^*_s}dx \leq \int_{|x| \geq R} |u_n|^{p^*_s}dx.$$

From (\ref{C-2}) we obtain,
\begin{equation}\label{C-12}
\mu_\infty=\lim_{R \to \infty} \limsup_{n \to \infty} \int_{\Rn}\psi_R^p|D^su_n|^pdx, \quad \nu_\infty=\lim_{R \to \infty} \limsup_{n \to \infty} \int_{\Rn}\psi_R^{p^*_s}|u_n|^{p^*_s}dx.
\end{equation}
Substituting (\ref{C-12}) into (\ref{C-11}) yields
$$S_{s,p} \nu_\infty^{p/{p^*_s}} \leq \mu_\infty.$$

{\bf Step 3}:  Assume $S_{s,p}\|\nu\|^{p/p^*_s}=\|\mu\|$. Then following the exact similar analysis as in \cite[Step 3, Lemma 1.40]{W} we get $\mu$ and $\nu$ are concentrated at a single point. 

\vspace{2mm}

{\bf Step 4}: For the general case write $v_n=u_n-u.$ Since $v_n \deb 0$ in $\dot{W}^{s,p}(\Rn),$ it follows $|D^s v_n|^p \deb \mu+|D^s u|^p$ in $\mathcal{M}(\Rn).$

Using Brezis-Lieb lemma, for all $h \in C_c(\Rn)$,  we obtain
$$\int_{\Rn}h|u|^{p^*_s}dx=\lim_{n\to\infty}\int_{\Rn}h|u_n|^{p^*_s}dx-\int_{\Rn}h|v_n|^{p^*_s}dx.$$
This in turn implies
$$|u_n|^{p^*_s} \deb |u|^{p^*_s}+\nu \quad\mbox{in}\quad \mathcal{M}(\Rn).$$
(\ref{C-3}) follows from corresponding inequality of $(v_n).$

\vspace{2mm}

{\bf Step 5}: Since ,
$$\limsup_{n \to \infty} \int_{|x|>R}|D^sv_n|^{p}dx=\limsup_{n \to \infty}\int_{|x|>R}|D^su_n|^{p}dx-\int_{|x|>R}|D^su|^{p}dx ,$$
we obtain $\displaystyle\mu_\infty=\lim_{R \to \infty} \limsup_{n \to \infty}\int_{|x|>R}|D^sv_n|^pdx.$

Similarly, applying Brezis-Lieb lemma to $\displaystyle\int_{|x|>R}|u|^{p^*_s}dx$ yields
$$\nu_\infty=\lim_{R \to \infty}\limsup_{n \to \infty}\int_{|x|\geq R}|v_n|^{p^*_s}dx.$$

Now, (\ref{C-4}) follows from corresponding inequality for $(v_n).$

\vspace{2mm}

{\bf Step 6}: For $R>1$, we have
\Bea
\limsup_{n \to \infty}\int_{\Rn}|D^su_n|^pdx&=&\limsup_{n \to \infty}\bigg(\int_{\Rn}\psi_R|D^su_n|^p +\int_{\Rn}(1-\psi_R)|D^su_n|^p \bigg)\\
&=&\limsup_{n \to \infty}\bigg(\int_{\Rn}\psi_R|D^su_n|^p \\
&&+\int_{\Rn}(1-\psi_R)d\mu+\int_{\Rn}(1-\psi_R)|D^su|^pdx\bigg).
\Eea
Hence, taking the limit $R \to \infty$ yields   
\Bea
\limsup_{n \to \infty} \int_{\Rn}|D^s u_n|^pdx=\mu_\infty+\|\mu\|+\|u\|_{0,s,p}^p.
\Eea
Proof of (\ref{C-6}) is similar.
\end{proof}

\section{Proof of Theorem \ref{Thm}}

The energy functional associated to $(P_{\theta,\la})$ is given by:
	\begin{equation}\label{I}
	I(u)=\frac{1}{p}\|u\|_{0,s_1,p}^p+\frac{1}{q}\|u\|_{0,s_2,q}^q-\frac{\theta}{r}\Iom V(x)|u|^rdx-\frac{1}{p^*_{s_1}}|u|_{p^*_{s_1}}^{p^*_{s_1}}-\la \Iom F(x,u)dx.
	\end{equation}
	We note that $I(u)=I(-u)$ for all $u \in X_{0,s_1,p}(\Om)$ and $I \in C^1(X_{0,s_1,p},\R).$  
	
 \begin{definition}
 	Let $\va\in C^1(X, \R)$. We say that $\{u_n\}$ is a Palais-Smale sequence (in short, PS sequence)  of $\va$ at level $c$ if $\va(u_n)\to c$ and $\va'(u_n)\to 0$ in $(X)'$, the dual space of $X$. Moreover,  we say that 
 	$\va$ satisfies (PS)$_c$ condition if $\{u_n\}$ is any (PS) sequence in $X$ at level $c$ implies $\{u_n\}$ has a convergent subsequence in $X$.
 \end{definition}

\begin{lemma}\label{lemma3}
Assume (A1)-(A3) are satisfied. Then, there exists $c_1,c_2>0$ such that any (PS)$_c$	
sequence $\{u_n\} \subset X_{0,s_1,p}(\Om)$ of $I$ has a convergent subsequence where 
$$c<\frac{s_1}{N}(S_{s_1,p})^\frac{N}{s_1p}-c_1\theta^{\frac{q}{q-r}}-c_2\la^{\frac{p^*_{s_1}}{p^*_{s_1}-l}}.$$ 
\end{lemma}

\begin{proof}
Let $\{u_n\} \subset X_{0,s_1,p}(\Om)$ be a (PS)$_c$ sequence of $I$. Therefore,
\begin{equation}\label{I'}
I(u_n)=c+o(1),\,\,I'(u_n)=o(1).
\end{equation}

\vspace{2mm}

{\bf Claim 1}: $\|u_n\|_{0,s_1,p}$ is uniformly bounded.\\

We prove the Claim by method of contradiction. Thus assume  the claim does not hold, that is, up to a subsequence $\|u_n\|_{0,s_1,p} \to \infty$ as 
$n \to \infty.$ Let us define $\hat{u_n}:=\frac{u_n}{\|u_n\|_{0,s_1,p}}.$ Then $\|\hat{u_n}\|_{0,s_1,p}=1.$
Therefore, up to a subsequence, we may take
\begin{equation}\label{hat-u-n}
\hat{u_n} \deb \hat{u} \quad\mbox{in}\quad X_{0,s_1,p}, \quad\text{and}\quad
\hat{u_n} \to \hat{u} \quad\mbox{in}\quad L^{\ga}(\Rn),\, 1 \leq \ga<p^*_{s_1}
\end{equation}
for some $\hat{u} \in X_{0,s_1,p}(\Om).$
From (\ref{I'}) using $\frac{1}{||u_n||_{0,s_1,p}}=o(1)$, we have
\begin{align}\label{+1}
&\frac{1}{p}\|\hat{u_n}\|_{0,s_1,p}^p+\frac{1}{q}\|u_n\|_{0,s_1,p}^{q-p}\|\hat{u_n}\|_{0,s_2,q}^q
-\frac{\theta}{r}\|u_n\|_{0,s_1,p}^{r-p}\Iom V(x)|\hat{u_n}|^rdx\no\\
&-\,\,\frac{1}{p^*_{s_1}}\|u_n\|_{0,s_1,p}^{p^*_{s_1}-p}|\hat{u_n}|_{p^*_{s_1}}^{p^*_{s_1}}-\la\|u_n\|_{0,s_1,p}^{-p}\Iom F(x,u_n)dx\no\\
&= o(1),
\end{align}
and
\begin{align}\label{+2}
\|\hat{u_n}\|_{0,s_1,p}^p+\|u_n\|_{0,s_1,p}^{q-p}\|\hat{u_n}\|_{0,s_2,q}^q-&\theta\|u_n\|_{0,s_1,p}^{r-p}\Iom V(x)|\hat{u_n}|^rdx\no\\
-&\|u_n\|_{0,s_1,p}^{p^*_{s_1}-p}|\hat{u_n}|_{p^*_{s_1}}^{p^*_{s_1}}-\la\|u_n\|_{0,s_1,p}^{-p}\Iom f(x,u_n)u_ndx
= o(1).
\end{align}
As $V \in L^{\infty}(\Om)$, using (\ref{hat-u-n}) we have
\begin{equation}\label{+3}
\Iom V(x)|\hat{u_n}|^r dx \to \Iom V(x)|\hat{u}|^r dx.
\end{equation}
From (\ref{+1}) and (\ref{+2}), we obtain
\begin{align}\label{+4}
(\frac{p^*_{s_1}}{p}-1)\|\hat{u_n}\|_{0,s_1,p}^p+&(\frac{p^*_{s_1}}{q}-1)\|u_n\|_{0,s_1,p}^{q-p}\|\hat{u_n}\|_{0,s_2,q}^q-
\theta(\frac{p^*_{s_1}}{r}-1)\|u_n\|_{0,s_1,p}^{r-p}\Iom V(x)|\hat{u_n}|^rdx\no\\
-&\la\|u_n\|_{0,s_1,p}^{-p}\bigg(p^*_{s_1}\Iom [F(x,u_n)-f(x,u_n)u_n]dx \bigg)=o(1).
\end{align}
Using (A3), (\ref{hat-u-n}) and (\ref{+3}), we can write
\Bea
\bigg(\frac{p^*_{s_1}}{p}-1\bigg)\|\hat{u_n}\|_{0,s_1,p}^p&=&\bigg(1-\frac{p^*_{s_1}}{p}\bigg)\|u_n\|_{0,s_1,p}^{q-p}\|\hat{u_n}\|_{0,s_2,q}^q \\
&+&\theta\bigg(\frac{p^*_{s_1}}{r}-1\bigg)\|u_n\|_{0,s_1,p}^{r-p}\Iom V(x)|\hat{u_n}|^r dx\\
&&+ \la\|u_n\|_{0,s_1,p}^{-p}\bigg(\Iom p^*_{s_1} F(x,u_n)-f(x,u_n)u_ndx \bigg)+o(1)\\
&\leq& \bigg(1-\frac{p^*_{s_1}}{p}\bigg)\|u_n\|_{0,s_1,p}^{q-p}\|\hat{u_n}\|_{0,s_2,q}^q\\
&+&\theta\bigg(\frac{p^*_{s_1}}{r}-1\bigg)\|u_n\|_{0,s_1,p}^{r-p}\Iom V(x)|\hat{u}|^rdx\\
&&+ \la a_3\|u_n\|_{0,s_1,p}^{l-p}|\hat{u}|_l^l+o(1)\\
&=&o(1),
\Eea
as $n \to \infty.$ This is a contradiction as $\|\hat{u_n}\|_{0,s_1,p}=1$ and hence  Claim 1 follows.

\vspace{2mm}

Consequently,  there exists $u \in X_{0,s_1,p}(\Om)$ such that up to a subsequence
\begin{align*}
&u_n \deb u\quad\mbox{in}\quad X_{0,s_1,p}(\Om),\\
&u_n \to u\quad\mbox{a.e. in}\quad \Rn,\\
&u_n \to u\quad\mbox{strongly in}\quad L^\gamma(\Rn)\quad \mbox{for}\quad 1 \leq \gamma <p^*_{s_1}.
\end{align*}
Applying (A1) and (A2), we have
$$\Iom f(x,u_n)u_ndx=\Iom f(x,u)u dx +o(1),$$
$$\Iom F(x,u_n)dx=\Iom F(x,u)dx+o(1),$$ and
$$\Iom V(x)|u_n|^rdx=\Iom V(x)|u|^rdx+o(1).$$
Note that by Lemma \ref{Lemma4},  $\|u_n\|_{0,s_2,q}$ is also bounded. Since $u_n \to u$ a.e. in $\Rn$, we obtain
$$\frac{|u_n(x)-u_n(y)|^{p-2}(u_n(x)-u_n(y))}{|x-y|^{(\frac{N}{p}+s_1)(p-1)}} \to 
\frac{|u(x)-u(y)|^{p-2}(u(x)-u(y))}{|x-y|^{(\frac{N}{p}+s_1)(p-1)}}$$
a.e. $(x,y) \in \Rn \times \Rn$. On the other hand, $\|u_n\|_{0,s_1,p}$ is uniformly bounded implies there exists $C>0$ such that
$$\int_{\R^{2N}} \bigg(\frac{|u_n(x)-u_n(y)|}{|x-y|^{\frac{N}{p}+s_1}}\bigg)^p dxdy \leq C \quad\mbox{for all}\quad n \geq 1,$$
that is,
$$\int_{\R^{2N}} \bigg|\frac{|u_n(x)-u_n(y)|^{p-2}(u_n(x)-u_n(y))}{|x-y|^{(\frac{N}{p}+s_1)(p-1)}}\bigg|^{\frac{p}{p-1}}dxdy \leq C.$$
Therefore, 
$$\frac{|u_n(x)-u_n(y)|^{p-2}(u_n(x)-u_n(y))}{|x-y|^{(\frac{N}{p}+s_1)(p-1)}} \deb \frac{|u(x)-u(y)|^{p-2}(u(x)-u(y))}{|x-y|^{(\frac{N}{p}+s_1)(p-1)}}$$
weakly in $L^{p'}(\Rn \times \Rn)$ with $p'=\frac{p}{p-1}.$ 
Similarly, as $\|u_n\|_{0,s_2,q}$ is uniformly bounded, 
$$\frac{|u_n(x)-u_n(y)|^{q-2}(u_n(x)-u_n(y))}{|x-y|^{(\frac{N}{q}+s_2)(q-1)}} \deb \frac{|u(x)-u(y)|^{q-2}(u(x)-u(y))}{|x-y|^{(\frac{N}{q}+s_2)(q-1)}}$$
weakly in $L^{q'}(\Rn \times \Rn)$ with $q'=\frac{q}{q-1}.$ 
If $\phi\in X_{0,s_1,p}(\Om),$  it follows $\frac{\phi(x)-\phi(y)}{|x-y|^{\frac{N}{p}+s_1}} \in L^p(\Rn \times \Rn)$ and
$\frac{\phi(x)-\phi(y)}{|x-y|^{\frac{N}{q}+s_2}} \in L^q(\Rn \times \Rn).$
As a result,
\Bea&&\int_{\R^{2N}}\frac{|u_n(x)-u_n(y)|^{p-2}(u_n(x)-u_n(y))(\phi(x)-\phi(y))}{|x-y|^{(\frac{N}{p}+s_1)(p-1)}|x-y|^{\frac{N}{p}+s_1}}dxdy\\
 &&\To \int_{\R^{2N}}\frac{|u(x)-u(y)|^{p-2}(u(x)-u(y))(\phi(x)-\phi(y))}{|x-y|^{N+s_1p}}dxdy \Eea
and 
\Bea&&\int_{\R^{2N}}\frac{|u_n(x)-u_n(y)|^{q-2}(u_n(x)-u_n(y))(\phi(x)-\phi(y))}{|x-y|^{(\frac{N}{q}+s_2)(q-1)}|x-y|^{\frac{N}{q}+s_2}}dxdy\\
&& \To \int_{\R^{2N}}\frac{|u(x)-u(y)|^{q-2}(u(x)-u(y))(\phi(x)-\phi(y)) }{|x-y|^{N+s_2q}}dxdy. \Eea
These together with (\ref{I'}) via Vitali's convergence theorem implies $I'(u)=0$ that is $u$ is weak solution of $(P_{\theta,\la})$. 

\vspace{2mm}

{\bf Claim 2}: $u_n \to u$ in $X_{0,s_1,p}(\Om).$ 

\vspace{2mm}

To prove this claim, define $v_n :=u_n-u.$ As $\|u_n\|_{0,s_1,p}$ and $\|u_n\|_{0,s_2,q}$
are uniformly bounded and $u_n \to u$ a.e. in $\Rn,$ applying Brezis-Lieb lemma,
 we obtain
\Bea
\int_{\R^{2N}}\frac{|u_n(x)-u_n(y)|^p}{|x-y|^{N+s_1p}}dxdy=\int_{\R^{2N}}\frac{|v_n(x)-v_n(y)|^p}{|x-y|^{N+s_1p}}dxdy
+\int_{\R^{2N}}\frac{|u(x)-u(y)|^p}{|x-y|^{N+s_1p}}dxdy+o(1),
\Eea
i.e., $\|u_n\|_{0,s_1,p}^p=\|v_n\|_{0,s_1,p}^p+\|u\|_{0,s_1,p}^p+o(1).$\\
Similarly, we have
$\|u_n\|_{0,s_2,q}^q=\|v_n\|_{0,s_2,q}^q+\|u\|_{0,s_2,q}^q+o(1).$
Therefore, a straight forward computation yields
\bea\label{+5}
c+o(1) &=&\frac{1}{p}\|v_n\|_{0,s_1,p}^p+\frac{1}{q}\|v_n\|_{0,s_2,q}^q-\frac{\theta}{r}\Iom V(x)|u|^rdx-\frac{1}{p^*}|v_n|_{p^*_{s_1}}^{p^*_{s_1}}\no\\
&&-\la\Iom F(x,u)dx
+\frac{1}{p}\|u\|_{0,s_1,p}^p+\frac{1}{q}\|u\|_{0,s_2,q}^q-\frac{1}{p^*}|u|_{p^*_{s_1}}^{p^*_{s_1}}.
\eea
On the other hand, using $|I'(u_n)u_n|\leq o(1)\|u_n\|_{0,s_1,p}=o(1)$,  we also have
\bea\label{+6}
\|v_n\|_{0,s_1,p}^p+\|v_n\|_{0,s_2,q}^q &=&o(1)+\theta\Iom V(x)|u|^rdx+|u|_{p^*_{s_1}}^{p^*_{s_1}}+|v_n|_{p^*_{s_1}}^{p^*_{s_1}}\no\\
&&+\la\Iom f(x,u)udx-\|u\|_{0,s_1,p}^p-||u||_{0,s_2,q}^q.
\eea
Combining (\ref{+6}) with $I'(u)=0$ yields
\begin{equation}\label{+7}
\|v_n\|_{0,s_1,p}^p+\|v_n\|_{0,s_2,q}^q-|v_n|_{p^*_{s_1}}^{p^*_{s_1}}=o(1).
\end{equation}
Since $\|v_n\|_{0,s_1,p},\, \|v_n\|_{0,s_2,q},\, |v_n|_{p^*_{s_1}}$ all are bounded sequence of real numbers, we may assume that:
\be\lab{25-11-17-1}\|v_n\|_{0,s_1,p}^p=a+o(1),\quad \|v_n\|_{0,s_2,q}^q=b+o(1),\quad |v_n|_{p^*_{s_1}}^{p^*_{s_1}}=d+o(1)\ee for some 
$a,b,d \geq 0.$ Hence, (\ref{+7}) implies
\begin{equation}\label{+8}
a+b=d.
\end{equation}
Thus $a\leq d$. Therefore, Sobolev inequality yields
\begin{equation}\label{+9}
a \geq S_{s_1, p} d^{p/{p^*_{s_1}}} \geq S_{s_1, p} a^{p/{p^*_{s_1}}}
\end{equation}
If $a=0,$ we are done.
If $a>0,$ then (\ref{+9}) implies
\begin{equation}\label{+10}
a \geq (S_{s_1, p})^\frac{N}{s_1p}.
\end{equation}
Using \eqref{+5}, \eqref{25-11-17-1}, \eqref{+8}, \eqref{+10} and the fact that $q<p<p^*_{s_1}$, taking the limit $n\to\infty$ we have
\bea\lab{25-11-17-2}
c&=&\frac{a}{p}+\frac{b}{q}-\frac{(a+b)}{p^*_{s_1}}+\frac{1}{p}\|u\|_{0,s_1,p}^p+\frac{1}{q}\|u\|_{0,s_2,q}^q-\frac{1}{p^*_{s_1}}|u|_{p^*_{s_1}}^{p^*_{s_1}}\no\\
&&-\frac{\theta}{r}\Iom V(x)|u|^r dx-\la \Iom F(x,u)dx.\no\\
&\geq&\frac{as_1}{N}+\frac{1}{p}\|u\|_{0,s_1,p}^p+\frac{1}{q}\|u\|_{0,s_2,q}^q-\frac{1}{p^*_{s_1}}|u|_{p^*_{s_1}}^{p^*_{s_1}}-\frac{\theta}{r}\Iom V(x)|u|^r dx-\la \Iom F(x,u)dx.\no\\
&\geq& \frac{s_1}{N}(S_{s_1, p})^\frac{N}{s_1p} +\frac{1}{p}\|u\|_{0,s_1,p}^p+\frac{1}{q}\|u\|_{0,s_2,q}^q-\frac{1}{p^*_{s_1}}|u|_{p^*_{s_1}}^{p^*_{s_1}}-\frac{\theta}{r}\Iom V(x)|u|^r dx-\la \Iom F(x,u)dx.\no\\
\,\,
\eea
Also from $<I'(u), u>=0$, it follows
\be\lab{25-11-17-3}
\|u\|_{0,s_1,p}^p=-\|u\|_{0,s_2,q}^q+|u|_{p^*_{s_1}}^{p^*_{s_1}}+\theta\Iom V(x)|u|^r\,dx+\la\Iom f(x,u)u\, dx.
\ee
Substituting \eqref{25-11-17-3} into \eqref{25-11-17-2} and using (A1) yields 
\bea\lab{25-11-17-4}
c&\geq& \frac{s_1}{N}(S_{s_1, p})^\frac{N}{s_1p}+\frac{s_1}{N}|u|_{p^*_{s_1}}^{p^*_{s_1}}-\theta\bigg(\frac{1}{r}-\frac{1}{p}\bigg)\Iom V(x)|u|^r dx\no\\
&&-\la\Iom (F(x,u)-\frac{1}{p}f(x,u)u)dx+\bigg(\frac{1}{q}-\frac{1}{p}\bigg)\|u\|_{0,s_2,q}^q\no\\
&\geq& \frac{s_1}{N}(S_{s_1, p})^\frac{N}{s_1p}+\frac{s_1}{N}|u|_{p^*_{s_1}}^{p^*_{s_1}}-\theta \eta\bigg(\frac{1}{r}-\frac{1}{p}\bigg)\|u\|_{0,s_2,r}^r\no\\
&&-\la\Iom (F(x,u)-\frac{1}{p}f(x,u)u)dx+\bigg(\frac{1}{q}-\frac{1}{p}\bigg)\|u\|_{0,s_2,q}^q.
\eea
Note that from (A4) it is easy to see $f(x,t)t\geq 0$ for all $t\in\R$, $x\in\Om$ and from (A3), it follows that $F(x,t)\leq \frac{1}{p^*_{s_1}}f(x,t)t+\frac{a_3}{p^*_{s_1}}|t|^l$. Thus, 
\be\lab{25-11-17-5}
\Iom \la\big(F(x,u)-\frac{1}{p}f(x,u)u\big)dx\leq \frac{\la a_3}{p}|u|_l^l\leq\frac{\la a_3}{p}|\Om|^{1-\frac{l}{p^*_{s_1}}}|u|_{p^*_{s_1}}^l =\la c_0|u|_{p^*_{s_1}}^l,
\ee
where $c_0=\frac{a_3}{p}|\Om|^{1-\frac{l}{p^*_{s_1}}}$. Applying Lemma \ref{Lemma4} and Young's inequality, for any $\de>0$ we  obtain 
\be\lab{25-11-17-6}
\eta\bigg(\frac{1}{r}-\frac{1}{p}\bigg)\|u\|_{0,s_2,r}^r\leq \eta\bigg(\frac{1}{r}-\frac{1}{p}\bigg)C^r\|u\|_{0,s_2,q}^r\leq \de\|u\|_{0,s_2,q}^q+C_{\de}.
\ee
Substituting \eqref{25-11-17-5} and \eqref{25-11-17-6} into \eqref{25-11-17-4} we have
\be\lab{25-11-17-7}
c\geq \frac{s_1}{N}(S_{s_1, p})^\frac{N}{s_1p}+\frac{s_1}{N}|u|_{p^*_{s_1}}^{p^*_{s_1}}-\theta \de\|u\|_{0,s_2,q}^q-\theta C_\de-\la c_0|u|_{p^*_{s_1}}^l+\bigg(\frac{1}{q}-\frac{1}{p}\bigg)\|u\|_{0,s_2,q}^q.\ee
Now choose $\de=\frac{1}{\theta}\big(\frac{1}{q}-\frac{1}{p}\big).$ This implies $C_{\de}=c_1\theta^\frac{r}{q-r}$, for some $c_1=c_1(p,q,r,N,s_1, s_2,|\Om|)>0$. Substituting this in \eqref{25-11-17-7} yields
\Bea
c&\geq&\frac{s_1}{N}(S_{s_1, p})^\frac{N}{s_1p}+\frac{s_1}{N}|u|_{p^*_{s_1}}^{p^*_{s_1}}-c_1\theta^{\frac{q}{q-r}}-\la c_0|u|_{p^*_{s_1}}^l.
\Eea
Note that the constants $c_1$ and $c_0$ are independent of $\theta,\la.$ 
Let us consider the function $g:(0,\infty) \to \R$ by
$$g(x)=\frac{s_1}{N}x^{p^*_{s_1}}-\la c_0 x^l.$$ We note that $g$ attains its minimum at $x_0=(\frac{c_0lN\la}{p^*_{s_1}})^{\frac{1}{p^*_{s_1}-l}}.$ Therefore, $$g(x)\geq g(x_0)=-c_2\la^{\frac{p^*_{s_1}}{p^*_{s_1}-l}},$$ where $c_2=c_0 \frac{p^*_{s_1}-l}{p^*_{s_1}}(\frac{c_0lN}{p^*_{s_1}})^{\frac{l}{p^*_{s_1}-l}}>0.$ Consequently,
$$c \geq \frac{s_1}{N}(S_{s_1, p})^\frac{N}{s_1p}-c_1\theta^{\frac{q}{q-r}}-c_2\la^{\frac{p^*_{s_1}}{p^*_{s_1}-l}},$$ which is a contradiction to the assumption on $c$. Hence, $a=0$ and this completes the proof of the lemma.

\end{proof}

Using (A1) and Lemma \ref{Lemma4}, for $1<r<p$


\be\lab{+11}
\Iom V(x)|u|^rdx\leq \eta\|u\|_{0,s_2,r}^r\leq C\eta\|u\|_{0,s_1,p}^r.
\ee

Moreover, by Sobolev embedding we have $S_{s_1, p}|u|_{p^*_{s_1}}^p \leq ||u||_{0,s_1,p}^p$ and using (A2), we obtain
\Bea
\Iom F(x,u)dx &\leq& \frac{a_1}{\al}|u|_{\al}^{\al} +\frac{a_2}{\ba}|u|_{\ba}^{\ba}\\
&\leq& \frac{a_1}{\al}|\Om|^{1-\frac{\al}{p^*_{s_1}}}|u|_{p^*_{s_1}}^{\al}+\frac{a_2}{\ba}|\Om|^{1-\frac{\ba}{p^*_{s_1}}}|u|_{p^*_{s_1}}^{\ba}\\
&\leq& 
\frac{a_1}{\al}|\Om|^{1-\frac{\al}{p^*_{s_1}}}(S_{s_1, p})^{-\al/p}\|u\|_{0,s_1,p}^{\al}+\frac{a_2}{\ba}|\Om|^{1-\frac{\ba}{p^*_{s_1}}}(S_{s_1, p})^{-\ba/p}\|u\|_{0,s_1,p}^{\ba}.\\
\Eea
This together with (\ref{+11}) and Sobolev embedding gives:
\bea\label{+12}
I(u)&\geq& \frac{1}{p}\|u\|_{0,s_1,p}^p-\frac{(S_{s_1, p})^{-{p^*_{s_1}}/p}}{p^*_{s_1}}\|u\|_{0,s_1,p}^{p^*_{s_1}}-\frac{\eta C\theta}{r}
\|u\|_{0,s_1,p}^r\no\\
&&- \la\frac{a_1}{\al}|\Om|^{\frac{p^*_{s_1}-\al}{p^*_{s_1}}}(S_{s_1, p})^{-\al/p}\|u\|_{0,s_1,p}^{\al}-\la \frac{a_2}{\ba}|\Om|^{\frac{p^*_{s_1}-\ba}{p^*_{s_1}}}(S_{s_1, p})^{-\ba/p}\|u\|_{0,s_1,p}^{\ba} \no\\
&=& c_3 \|u\|_{0,s_1,p}^p-c_4\|u\|_{0,s_1,p}^{p^*_{s_1}}-c_5\theta\|u\|_{0,s_1,p}^r-c_6\la\|u\|_{0,s_1,p}^{\al}
-c_7\la\|u\|_{0,s_1,p}^{\ba},
\eea 
where $$c_3=\frac{1}{p},\, c_4=\frac{(S_{s_1, p})^{-{p^*_{s_1}}/p}}{p^*_{s_1}},\, c_5=\frac{\eta}{r}C
,\, c_6=\frac{a_1}{\al}|\Om|^{\frac{p^*_{s_1}-\al}{p^*_{s_1}}}(S_{s_1, p})^{-\al/p},\, c_7=\frac{a_2}{\ba}|\Om|^{\frac{p^*_{s_1}-\ba}{p^*_{s_1}}}(S_{s_1, p})^{-\ba/p}$$ 
are all positive constants. 
Let us define a function $h:(0,\infty) \to \R$ by
\begin{equation}\label{h}
h(x)=c_3x^p-c_4x^{p^*_{s_1}}-c_5\theta x^r-c_6\la x^{\al}-c_7 \la x^{\ba}.
\end{equation}
As $1<r<p$ and $1<\al,\, \ba<p^*_{s_1},$ we see that there exists $\la_0 \geq \la^*>0$ such that for any $\la \in (0,\la^*),$ there exists $x>0$ such that $h(x)>0.$ Therefore, we conclude that for any $\la \in (0,\la^*),$ there exists
\begin{equation}\label{theta-*}
\theta^*=\theta^*(\la)>0
\end{equation}  such that for any $\theta \in (0,\theta^*),$
\begin{itemize}
	\item [(a)]
	$h(x)$ attains its maximum and $\max_{x \in (0,\infty)}h(x)>0$,
	\item [(b)]
	$\frac{s_1}{N}S^\frac{N}{s_1p}-c_1 \theta^{\frac{q}{q-r}}-c_2 \la^{\frac{p^*_{s_1}}{p^*_{s_1}-l}}>0$,
	\end{itemize}
	 where $c_1,c_2$ are given in lemma \ref{lemma3}. From the definition of $h$, it is not difficult to see that $h$ has finitely many positive roots, say $0<r_1<r_2<\cdots<r_m<\infty$, where $h(r_i)=0$.

As a result, we note that,
\begin{equation}\label{h-0}
h(x) \begin{cases}<0 \,\,\ \forall\, x \in (0,r_1)\cup(r_2,r_3)\cup \cdots \cup(r_m,\infty), \\
>0 \,\,\ \forall\, x \in (r_1,r_2)\cup(r_3,r_4)\cup \dots \cup (r_{m-1},r_m).
\end{cases}
\end{equation}
Denote,
$$A:=(0,r_1)\cup(r_2,r_3)\cup \dots \cup(r_m,\infty), \quad B:=A \setminus (r_m,\infty).$$ We choose $\tau \in C^{\infty}(\R^+;[0,1])$ such that
\begin{equation}\lab{27-11-17-1'}
\tau(x)=\begin{cases}
        1,\,\ x \in B,\\
        0,\,\ x \in (r_m,\infty).
        \end{cases}
\end{equation}
Set $\phi(u):=\tau(\|u\|_{0,s_1,p})$ and the truncated functional 
\begin{equation}\label{+13}
I_\infty(u)=\frac{1}{p}\|u\|_{0,s_1,p}^p+\frac{1}{q}\|u\|_{0,s_2,q}^q-\frac{\theta}{r}\Iom V(x)|u|^rdx-\frac{1}{p^*_{s_1}}\Iom |u|^{p^*_{s_1}}\phi(u)dx \\
-\la \Iom F(x,u)\phi(u)dx.
\end{equation}
Similarly, as (\ref{h}) we can consider the function $\bar{h}:(0,\infty) \to \R$ as
\begin{equation}\label{+14}
\bar{h}(x)=c_3x^p-c_4x^{p^*_{s_1}}\tau(x)-c_5\theta x^r-c_6 \la x^{\al}\tau(x)-c_7\la x^{\ba}\tau(x),\,\,\ \forall\, x>0
\end{equation}
and have
\begin{equation}\label{+15}
I_\infty(u) \geq \bar{h}(\|u\|_{0,s_1,p}).
\end{equation}

It is not difficult to check that from the definition of $\tau,\, A,\, B$ that 
\be\lab{26-11-17-1} \bar{h}(x) \geq h(x) \,\, \forall\, x>0,\quad
 \bar{h}(x)=h(x) \,\,\forall\, x \in B,\quad \bar{h}(x) \geq 0 \,\, \forall\, x>r_m.\ee 
 
 Therefore, we conclude 
\be\lab{26-11-17-3}I(u)=I_\infty(u) \quad\text{for}\quad \|u\|_{0,s_1,p} \in B.\ee
Also we note that  $\tau \in C^{\infty}(\R^+,[0,1])$ implies $I_\infty(u) \in C^1(X_{0,s_1,p},\R).$ 

 \begin{lemma}\label{Lemma5}
 (i) Let $I_\infty(u)<0.$ Then $\|u\|_{0,s_1,p} \in B$ and there exists a neighbourhood $\mathcal{N}_u$ of $u$ such that $I(v)=I_\infty(v) \,\,\ \forall\, v \in \mathcal{N}_u.$ 

 (ii) For any $\la \in (0,\la^*),$ there exists $\theta^*>0$ such that for any $\theta \in (0,\theta^*),$ $I_\infty(u)$ satisfies (PS)$_c$ condition for $c<0.$ 
 \end{lemma}
 \begin{proof}
 We prove (i) by method of contradiction. Suppose $\|u\|_{0,s_1,p} \notin B,$ that is, $\|u\|_{0,s_1,p} \in \R^+ \setminus B$ for $u$ with $I_\infty(u)<\infty.$
 Now, two cases may happen.\\
\textit{Case 1}: If $\|u\|_{0,s_1,p} \in \R^+ \setminus A,$ then using 
(\ref{+15}), \eqref{26-11-17-1} and (\ref{h-0}),  we have
$$I_\infty(u) \geq \bar{h}(\|u\|_{0,s_1,p})\geq h(\|u\|_{0,s_1,p})> 0.$$ This contradicts $I_{\infty}(u)<0$.

\textit{Case 2}: If $\|u\|_{0,s_1,p} \in (r_m,\infty)=A \setminus B$. Then by (\ref{+15}) and \eqref{26-11-17-1}, we have
$I_\infty(u) \geq \bar{h}(\|u\|_{0,s_1,p}) \geq 0,$ which again contradicts $I_{\infty}(u)<0$. Hence, $\|u\|_{0,s_1,p}\in B$. Moreover as $B$ is an open set, applying \eqref{26-11-17-3}, we obtain there exists a neighborhood $\mathcal{N}_u$ of $u$ such that $I(v)=I_\infty(v) \,\,\ \forall\, v \in \mathcal{N}_u.$ 

To prove (ii), let $\theta^*>0$ be as in (\ref{theta-*}). Suppose $c<0$ and $\{u_n\} \subseteq X_{0,s_1,p}(\Om)$ is a (PS)$_c$ sequence of $I_\infty.$ Therefore, for $n$ large we may take 
$$I_\infty(u_n)<0 \quad\mbox{and}\quad I'_\infty(u_n)=o(1).$$
Using (i) it follows that $\|u_n\|_{0,s_1,p} \in B.$ Therefore, $I(u_n)=I_\infty(u_n)$ and $I'(u_n)={I'}_\infty(u_n)=o(1).$ 
Since (b) holds for $\theta \in (0,\theta^*),$ applying lemma \ref{lemma3}, we obtain $I(u)$ satisfies (PS)$_c$ condition for $c<0.$ Therefore, $I_\infty(u)$ satisfies (PS)$_c$ condition for $c<0$.
 \end{proof}
 
 \vspace{2mm}
 
Define, 
\be\lab{26-11-17-7}\Sigma:=\{A \subset X_{0,s_1,p} \setminus\{0\}: A \quad\mbox{is closed},\quad A=-A\}.\ee

\begin{definition}
Let $A\in\Sigma$. We denote by $\ga(A)$ the genus of $A$ which is the smallest positive integer $n$ such that there exists an odd continuous map from $A$ into $\R^n\setminus\{0\}$. We set $\ga(\emptyset)=0$ and if no such $n$ exists for $A$, then we set $\ga(A)=\infty$.
\end{definition}

{\bf Proof of Theorem \ref{Thm}} \begin{proof}
Define $$c_k:=\inf_{A \in \Sigma_k} \sup_{A} I_{\infty}(u),$$
where 
$$\Sigma_k:=\{A \in \Sigma: \ga(A) \geq k\},$$ and $\Sigma$ is as in \eqref{26-11-17-7} .
Let, 
 $$K_c:=\{u \in X_{0,s_1,p}(\Om):\, I_\infty(u)=c,\,\, I'_\infty(u)=0\}$$ 
 and $\theta^*$ be as in (\ref{theta-*}) and  $\theta \in (0,\theta^*).$\\
 
{\bf Claim}: If $k,l \in \N$  such that $c_k=c_{k+1}=\dots=c_{k+l}=c$, then $c<0$ and $\ga(K_c) \geq l+1.$ 
 
 Let us consider the set
 $$I_\infty^{-\eps}:=\{u \in X_{0,s_1,p}(\Om):\, I_\infty(u) \leq -\eps\}.$$
 We will show that for any $k \in \N,$ there exists $\eps=\eps(k)>0$ such that $\ga(I_\infty^{-\eps}(u)) \geq k.$ 
 Fix $k \in \N.$ Let $X_k$ be a $k-$dimensional subspace of $X_{0,s_1,p}.$ Take $u \in X_k$ with $\|u\|_{0,s_1,p}=1.$ Thus for $0<\rho<r_1,$ using \eqref{26-11-17-3} we have
 \bea\label{I-infty}
 I(\rho u)=I_\infty(\rho u)&=&\frac{1}{p}\rho^p+\frac{\rho^q}{q}\|u\|_{0,s_2,q}^q-\frac{\theta \rho^r}{r}\Iom V(x)|u|^rdx\no\\
&-& \frac{\rho^{p^*}}{p^*}|u|_{p^*_{s_1}}^{p^*_{s_1}}-\la \Iom F(x,\rho u)dx.
 \eea
As $X_k$ is a finite dimensional subspace of $X_{0,s_1,p}(\Om),$ all norms in $X_k$ are equivalent and therefore
\begin{equation}\label{al-k}
\al_k:=\sup\{\|u\|_{0,s_2,q}^q :u \in X_k, \|u\|_{0,s_1,p}=1\}<\infty,
\end{equation}
 \begin{equation}\label{ba-k}
 \ba_k:=\inf\{|u|_{p^*_{s_1}}^{p^*_{s_1}}: u \in X_k,||u||_{0,s_1,p}=1\}>0,
 \end{equation}
 \begin{equation}\label{ga-k}
 \ga_k:=\inf\{|u|_r^r : u \in X_k,\|u\|_{0,s_1,p}=1\}>0.
 \end{equation}

Since using (A4), it follows that $F(x,\rho u)>0$, applying (\ref{I-infty})-(\ref{ga-k}), we obtain
$$I_\infty(\rho u)\leq \frac{1}{p}\rho^p+\al_k\frac{\rho^q}{q}-\sigma \ga_k \frac{\theta  \rho^r}{r}-\ba_k\frac{\rho^{p^*_{s_1}}}{p^*_{s_1}}.$$
 For any $\eps>0$, there exists $\rho\in(0,r_1)$ such that $I_\infty(\rho u) \leq -\eps$ for $u \in X_k$ with $\|u\|_{0,s_1,p} =1$. Define, $S_\rho=\{u \in X_{0,s_1,p}: \|u\|_{0,s_1,p}=\rho\}$. Then $S_\rho \cap X_k \subseteq I_\infty^{-\eps}$.   By Lemma \ref{l-ii}, it follows that
 $$ k=\ga(S_\rho \cap X_k)\leq \ga(I_\infty^{-\eps}).$$
 Therefore, we conclude $I_\infty^{-\eps} \in \Sigma_k$, since $I_\infty$ is continuous and even.  Consequently,
 \be\lab{4-1-1}c=c_k \leq \sup_{I_\infty^{-\eps} } I_{\infty}(u)\leq -\eps <0.\ee Note that by \eqref{+15} and \eqref{26-11-17-1}, we have $I_{\infty}(u)\geq h(\|u\|_{0,s_1,p})$, for all $u\in X_{0,s_1,p}$. Consequently, using \eqref{h-0} and \eqref{27-11-17-1'} in the definition of  $I_\infty$, it follows that $I_\infty$ is bounded from below. Thus $c=c_k>-\infty.$
By Lemma \ref{Lemma5}, $I_\infty$ satisfies (PS)$_c$ condition. We note that $K_c$ is a compact set. To see this, let $\{u_n\}$ be a sequence in $K_c.$ Then $I_\infty(u_n)=c$ and $I'_\infty(u_n)=0.$ Thus, $$\lim_{n \to \infty}I_\infty(u_n)=c,\lim_{n \to \infty}I'_\infty(u_n)=0.$$ Therefore, $\{u_n\}$ is a (PS)$_c$ sequence in $K_c.$ As $c<0,$ by Lemma \ref{Lemma5}, there exists a subsequence and $u \in X_{0,s_1,p}(\Om)$ such that $u_{n_k} \to u$ in $X_{0,s_1,p}(\Om)$ and $I_\infty(u)=c,\, I'_\infty(u)=0.$ As a result, $u \in K_c,$ that is, $\{u_n\}$ has a convergent subsequence in $K_c.$ 

Now let us complete the proof of our claim. Suppose the claim is not true, that is, $\ga(K_c) \leq l.$ Then, by Lemma \ref{l-ii}, there exists a neighbourhood of $K_c$, say $N_{r}(K_c)$ such that $\ga(N_{r}(K_c)) \leq l.$ Since $c<0,$ we may consider $N_{r}(K_c) \in I_\infty^0.$ By Lemma \ref{l-i}, there exists an odd homeomorphism $\bar{\eta}:X_{0,s_1,p}(\Om) \to X_{0,s_1,p}(\Om)$ such that
$$ \bar{\eta}(I_\infty^{c+\de}\setminus N_{r}(K_c)) \subset I_\infty^{c-\de}\quad\mbox{for some}\quad 0<\de<-c.$$
From the definition of $c=c_{k+l},$ we know there exists an $A \in \Sigma_{k+l}$ such that
$$\sup_{u \in A} I_\infty(u) <c+\de,$$ that is, $A \subset I_\infty^{c+\de}$ and
$$\bar{\eta}\big(A\setminus N_{r}(K_c)\big)\subset \bar{\eta}\big(I_\infty^{c+\de}\setminus N_{r}(K_c)\big)\subset I_\infty^{c-\de}.$$
This yields us:
\begin{equation}\label{sup-I-infty}
\sup_{u \in \bar{\eta}(A\setminus N_{r}(K_c))}I_\infty(u) \leq c-\de.
\end{equation}
 Again, by Lemma \ref{l-ii}, we have,
 $$\ga(\bar{\eta}(\overline{A\setminus N_{r}(K_c)})) = \ga(\overline{A\setminus N_{r}(K_c)})\geq \ga(A)-\ga(N_{r}(K_c))\geq k+l-l=k.$$
Therefore, we have $\bar{\eta}(\overline{A\setminus N_{r}(K_c)}) \in \Sigma_k$ and 
$\sup_{u \in \eta(\overline{A\setminus N_{r}(K_c)})}I_\infty(u)\geq c_k=c.$
This is a contradiction to (\ref{sup-I-infty}). Hence, we have the claim.\\

Now let us complete the proof of Theorem \ref{Thm}. Since $\Sigma_{k+1} \subseteq \Sigma_k,$ we have $c_k \leq c_{k+1} \,\, \forall\,\, k.$  If all $c_k$'s are distinct then $\ga(K_{c_k})\geq 1$, since $K_{c_k}$ is a compact set and by Lemma \ref{l-ii} (7), genus of a compact set is finite. Therefore, in that case  $I_{\infty}$ has infinitely many distinct critical points. If for some $k$, there exists $l$ such that $c_k=c_{k+1}=\cdots=c_{k+l}=c$, then by the above claim, $\ga(K_c)\geq l+1$ and therefore $K_c$ has infinitely many distinct elements, i.e, $I_{\infty}$ has infinitely many distinct critical points. Hence combining \eqref{4-1-1} along with Lemma \ref{Lemma5}, we conclude that $I$ has infinitely many distinct critical points. 
\end{proof}

 \section{proof of Theorem \ref{thm1}}
First, we consider the problem 
   \begin{equation}
 (\tilde P)\label{tl-P}
 \left\{\begin{aligned}
 (-\De)^{s_1}_p u + (-\De)^{s_2}_q u &=\theta(u^+)^{r-1} +(u^+)^{p^*_{s_1}-1}  \quad\text{in }\quad \Om, \\
 u=0& \quad\text{in}\quad \Rn \setminus \Omega.
 \end{aligned}
 \right.
 \end{equation}	

  \begin{definition}\label{sol} We say that $u\in X_{0,s_1,p}(\Om)$ is a weak solution of $(\tilde P)$
 if for all $\phi \in X_{0,s_1,p}$ we have, 
 	\Bea
 		&&\int_{\R^{2N}}\frac{|u(x)-u(y)|^{p-2}(u(x)-u(y))(\phi(x)-\phi(y))}{|x-y|^{N+ps_1}}dxdy \\
 		&&+ \int_{\R^{2N}}\frac{|u(x)-u(y)|^{q-2}(u(x)-u(y))(\phi(x)-\phi(y))}{|x-y|^{N+qs_2}}dxdy \\
 		&=& \theta\Iom (u(x)^{+})^{r-1}\phi(x)dx
 		+\Iom (u(x)^{+})^{p^*_{s_1}-1}\phi(x)dx .
 	\Eea
 	\end{definition}

The Euler-Lagrange energy functional associated to  $(\tilde P)$ is
 \bea\lab{I-theta}
 I_\theta(u) &=& \frac{1}{p}\|u\|^p_{0,s_1,p}+\frac{1}{q}\|u\|_{0,s_2,q}^q-\frac{\theta}{r}\Iom (u^{+})^r dx -\frac{1}{p^*_{s_1}}\Iom (u^+)^{p^*_{s_1}}dx.
 \eea
It can be checked that $I_\theta \in C^2(X_{0,s_1,p},\R)$ and any critical points of $I_\theta$ is a weak solution of $(\tilde P)$ and conversely.

We define, \be\lab{24th Jan}c_\theta=\inf_{u \in N_\theta}I_{\theta}(u),$$ where 
$$N_\theta:=\{u \in X_{0,s_1,p}(\Om)\setminus \{0\} : \<I'_{\theta}(u),u\>=0 \}.\ee 

We will show that $I_\theta$ has the Mountain Pass (MP) Geometry. 

\begin{lemma}\label{lem-1}
  Let $1<q<p<r<p^*_{s_1}.$ Then for any $\theta>0,$ 
  \begin{itemize}
  	\item [(a)]
  	there exist constants $\rho,\beta>0$ such that $I_\theta(u)>\beta$ for all $u \in X_{0,s_1,p}(\Om)$ with $\|u\|_{0,s_1,p}=\rho,$
  	\item [(b)]
  	there exist $u_0 \in X_{0,s_1,p}(\Om)$ such that $I_\theta(u_0)<0$ and $\|u_0\|_{0,s_1,p}>\rho.$ 
  \end{itemize}
\end{lemma}
\begin{proof}
Using Sobolev inequality and H\"older inequality in the definition of $I_{\theta}$, we obtain
\Bea
I_{\theta}(u)&\geq&\frac{1}{p}\|u\|_{0,s_1,p}^p+\frac{1}{q}\|u\|_{0,s_2,q}^q-\frac{\theta}{r}|\Omega|^{\frac{p^*_{s_1}-r}{p^*_{s_1}}}|u^+|_{p^*_{s_1}}^r-\frac{1}{p^*_{s_1}}|u^+|^{p^*_{s_1}}_{p^*_{s_1}}\\
&\geq&\frac{1}{p}\|u\|_{0,s_1,p}^p+\frac{1}{q}\|u\|_{0,s_2,q}^q-\frac{\theta}{r}|\Omega|^{\frac{p^*_{s_1}-r}{p^*_{s_1}}}S_{s_1,p}^{\frac{-r}{p}}\|u\|_{0,s_1,p}^r-\frac{1}{p^*_{s_1}}S_{s_1,p}^{\frac{-p^*_{s_1}}{p}}\|u\|^{p^*_{s_1}}_{0,s_1,p}.
\Eea	
As $1<q<p<r<p^*_{s_1}$, there exist two constants $\rho, \beta>0$ such that $I_{\theta}(u)>\beta$ for all $u \in X_{0,s_1,p}$ with $\|u\|_{0,s_1,p}=\rho$ and that proves (a).
 
To prove (b), we fix $u \in X_{0,s_1,p}(\Om)$ with  $u^+ \nequiv 0.$ Then it is easy to see that $\lim_{t \to +\infty}I_\theta(tu)=-\infty$. Thus we can choose $t_0>0$ such that 
 $\|t_0 u\|_{0,s_1,p}>\rho$ and $I_{\theta}(t_0 u)<0.$  Hence (b) holds.
\end{proof}

Define, \be\lab{1-12-17-2}C_\theta:=\inf_{u \in X_{0,s_1,p} \setminus \{0\}} \sup_{t \geq 0}I_\theta(tu).\ee

\begin{lemma}\label{lem-2}
Let $1<q<p<r<p^*_{s_1}.$ Then for any $\theta>0$, $I_\theta$ satisfies the $(PS)_c$ conditions for all $c \in \big(0,\frac{s_1}{N}(S_{s_1,p})^\frac{N}{s_1p}\big)$. Furthermore, there exists $\theta^*>0$ such that $$C_{\theta} \in \bigg(0,\frac{s_1}{N}(S_{s_1,p})^\frac{N}{s_1p}\bigg) \quad\mbox{for}\quad \theta>\theta^*.$$	
\end{lemma}

\begin{proof}
Let $c \in \big(0,\frac{s_1}{N}(S_{s_1,p})^\frac{N}{s_1p}\big) $ and 
$\{u_n\}_{n \geq 1} \subset X_{0,s_1,p}(\Om)$ be a $(PS)_c$ sequence of $I_\theta(\cdot)$. From Claim 1 in the proof of Lemma \ref{lemma3}, it follows that $\{u_n\}$ is uniformly bounded in  $X_{0,s_1,p}(\Om)$.  Therefore, there exists $u \in X_{0,s_1,p}(\Om)$ such that up to a subsequence,  $u_n \deb u$ in $X_{0,s_1,p}(\Om)$ and $u_n \to u$  in $L^{\ga}(\Om)$ for $1\leq \ga<p^*_{s_1}$  and $u_n \to u$ a.e. in $\Rn.$ Also, following the same arguments as in the proof of Lemma \ref{lemma3}, we see that $u$ is a critical point of $I_\theta$, that is $\langle I_{\theta}'(u), \phi\rangle=0$. Next, to prove $u_n\to u$ strongly in $X_{0,s_1,p}(\Om)$, we follow the arguments along the same line as in the proof of claim 2 of Lemma \ref{lemma3} and obtain either $\|u_n-u\|_{0,s_1,p}=o(1)$ or \eqref{25-11-17-4} holds with $\la=0$. Thus in the second case,
$$c\geq \frac{s_1}{N}(S_{s_1,p})^\frac{N}{s_1p}+\frac{s_1}{N}|u^+|_{p^*_{s_1}}^{p^*_{s_1}}+\theta \eta\bigg(\frac{1}{p}-\frac{1}{r}\bigg)\|u^+\|_{0,s_2,r}^r+\bigg(\frac{1}{q}-\frac{1}{p}\bigg)\|u\|_{0,s_2,q}^q\geq \frac{s_1}{N}(S_{s_1,p})^\frac{N}{s_1p}.$$ This contradicts the fact that $c \in \big(0,\frac{s_1}{N}(S_{s_1,p}\big)^\frac{N}{s_1p})$. Hence $\|u_n-u\|_{0,s_1,p}=o(1)$.
 Therefore, $I_\theta$ satisfies $(PS)_c$ condition for $c \in (0,\frac{s_1}{N}(S_{s_1,p})^\frac{N}{s_1p}).$

 Next, to prove  $C_\theta\in (0,\frac{s_1}{N}(S_{s_1,p})^\frac{N}{s_1p})$ we choose $u_0 \in X_{0,s_1,p}(\Om)$ with $u_0^- \equiv 0$ and 
 
 $|u_0|_{p^*_{s_1}}=1.$
 As $\lim_{t \to \infty}I_\theta(tu_0)=-\infty$ and $\lim_{t \to 0}I_\theta(tu_0)=0,$  there exists $t_\theta>0$ such that $\sup_{t \geq 0}I_\theta(tu_0)=I_\theta(t_\theta u_0).$ 
Therefore,
 $$t_\theta^{p-1}\|u_0\|_{0,s_1,p}^p+t_\theta^{q-1}\|u_0\|_{0,s_2,q}^q-\theta t_\theta^{r-1}|u_0|^r_r-t_\theta^{p^*_{s_1}-1}=0.$$
 So, we get,
 $t_\theta^{p-r}\|u_0\|_{0,s_1,p}^p+t_\theta^{q-r}\|u_0\|_{0,s_2,q}^q-t_\theta^{p^*_{s_1}-r}=\theta|u_0|_r^r.$
 As $1<q<p<r<p^*_{s_1},$ we get $t_\theta \to 0$ as $\theta \to \infty.$ 
Thus, there exists $\theta^*>0$ such that for any $\theta>\theta^*$ we have,
 $$\sup_{t \geq 0}I_\theta(tu_0)<\frac{s_1}{N}(S_{s_1,p})^\frac{N}{s_1p}. $$ Hence, $C_\theta \in \big(0,\frac{s_1}{N}(S_{s_1,p}\big)^\frac{N}{s_1p})$ for $\theta>\theta^*.$
\end{proof}
 
 \vspace{2mm}

{\bf Proof of theorem \ref{thm1}}: Using Lemma \ref{lem-1}, Lemma \ref{lem-2} and Lemma \ref{lem-10}, we conclude that $I_\theta$ has a critical point  $u\in X_{0,s_1,p}$ for $\theta>\theta^*$ where $\theta^*$ is given in (\ref{theta-*}).

{\bf Claim:} $u\geq 0$ almost everywhere.

Indeed,  \bea\label{H1'}
0= \<I_\theta'(u) ,u^-\>&=&\int_{\R^{2N}}\frac{|u(x)-u(y)|^{p-2}(u(x)-u(y))(u^-(x)-u^-(y))}{|x-y|^{N+s_1p}}dxdy\no\\
 &+&\int_{\R^{2N}}\frac{|u(x)-u(y)|^{q-2}(u(x)-u(y))(u^-(x)-u^-(y))}{|x-y|^{N+s_2q}}dxdy\no\\
 &:=&K_1+K_2,
 \eea
  Note that,
 \bea\label{H2'}
 (u(x)-u(y))(u^-(x)-u^-(y))
  &=&-u^+(y)u^-(x)-u^+(x)u^-(y)-(u^-(x)-u^-(y))^2\no\\
 &\leq&-(u^-(x)-u^-(y))^2 \leq 0
 \eea
 and
 \be\label{H3'}
|u(x)-u(y)|=\big(|u(x)-u(y)|^2\big)^\frac{1}{2} \geq\big( |u^-(x)-u^-(y)|^2\big)^\frac{1}{2}=|u^-(x)-u^-(y)|.
 \ee
 Since $2\leq q<p$, using (\ref{H2'}) and (\ref{H3'}), we obtain
 $$K_2\leq -\int_{\R^{2N}}\frac{|u^-(x)-u^-(y)|^q}{|x-y|^{N+s_2q}}dxdy=-\|u^-\|_{0,s_2,q}^q.$$
 Similarly, $K_1 \leq -\|u^-\|_{0,s_1,p}^p.$
 Therefore, (\ref{H1'}) implies, $\|u^-\|_{0,s_1,p}^p+\|u^-\|_{0,s_2,q}^q\leq 0$ that is, $u^-=0$ a.e and this proves the claim.

Further, we observe that  $C_\theta>0$, since $I_\theta$ satisfies the mountain pass geometry. Therefore, as $u$ is the critical point corresponding to $C_\theta$, $u$ must be nontrivial. Thus, $u$ is nontrivial nonnegative solution of $(\tilde P)$. Consequently,  $u$ is nontrivial nonnegative solution of $(P)$.

\section {Proof of Theorem \ref{thm2}}
We break the proof of Theorem \ref{thm2} into several lemmas. For the rest of the section, we assume  
\be\lab{1-12-17-3}
 N>p^2s_1 \quad\text{and}\quad 2\leq q<\frac{N(p-1)}{N-s}<p \leq \max \{p,\, p^*_{s_1}-\frac{q}{q-1}\}<r<p^*_{s_1}.
 \ee

Let $U$ be a radially symmetric and decreasing minimizer for the Sobolev constant defined in \eqref{S} for $s=s_1$ and it is known from \cite{BMS} that there exists constants $c_1,\, c_2>0$ and $\theta>1$ such that
\be\lab{1-12-17-1}
\frac{c_1}{|x|^{\frac{N-s_1p}{p-1}}}\leq U(|x|)\leq\frac{c_2}{|x|^\frac{N-s_1p}{p-1}} \quad \forall\,\, |x|\geq 1,
\ee
\begin{equation}\label{U-theta-r}
\frac{U(\theta r)}{U(r)} \leq \frac{1}{2} \quad\forall\,\, r\geq 1.
\end{equation}
Multiplying $U$ by a positive constant if necessary, we may assume that $U$ satisfies  the following:
\be\lab{27-12-9} (i)\,\,  (-\De)^{s_1}_p U=U^{p^*_{s_1}-1} \quad (ii)\,\, \|U\|_{0,s_1,p}^p=|U|_{p^*_{s_1}}^{p^*_{s_1}}=(S_{s_1,p})^{N/s_1p}.\ee
 For any $\de>0,$ the function $$U_\de(x)=\frac{1}{\de^{\frac{N-s_1p}{p}}}U(\frac{|x|}{\de})$$ is also a minimizer for $
S_{s_1,p}$ satisfying (i) and (ii).
Let $\theta$ be the universal constant defined as in (\ref{U-theta-r}).  We may assume without loss of generality that $0\in\Om$. For $\de, \, R>0$,  we define some auxiliary functions as in \cite{MPSY}. 

 $m_{\de,R}:=\frac{U_\de(R)}{U_\de(R)-U_\de(\theta R)},$ and $g_{\de,R}:[0,+\infty) \to \R$ by
 \begin{equation}
 g_{\de,R}(t)=\begin{cases}
 0, \quad \quad \quad \quad \quad \quad \quad 0\leq t \leq U_\de(\theta R)\\
 m_{\de,R}^p(t-U_\de(\theta R)),\,\  U_\de(\theta R) \leq t \leq U_\de(R)\\
 t+U_\de(R) (m_{\de,R}^{p-1}-1),\,\  t \geq U_\de(R),
 \end{cases}
 \end{equation}
 and $G_{\de,R}:[0,\infty) \to \R$ by
 \begin{equation}
 G_{\de,R}(t)=\int_0^t (g'_{\de,R}(\tau))^{1/p}d \tau=\begin{cases}
 0, \quad \quad \quad \quad \quad \quad \quad  0\leq t \leq U_\de(\theta R)\\
 m_{\de,R}(t-U_\de(\theta R)),\,\ U_\de(\theta R) \leq t \leq U_\de(R)\\
 t,\quad \quad \quad \quad \quad \quad \quad  t \geq U_\de(R).
 \end{cases}
 \end{equation}
 We note that $g_{\eps,\de}$ and $G_{\de,R}$ are non-decreasing and absolutely continuous.  Note that by definition,
 \begin{equation*}
 G'_{\de,R}(t)=\big(g'_{\de,R}(t)\big)^{\frac{1}{p}}=\begin{cases}
 0, \quad  \quad \quad 0\leq t < U_\de(\theta R)\\
 m_{\de,R},\,\  U_\de(\theta R) < t < U_\de(R)\\
 1,\,\, \qquad  t > U_\de(R),
 \end{cases}
 \end{equation*}
Therefore, \be\label{EQ3}G'_{\de,R}(t) \leq \max\{m_{\de,R},1\} \leq m_{\de,R}+1.\ee Next, we estimate  $m_{\de,R}$ as follows
 \be\label{EQ4}
 m_{\de,R}=\frac{U_\de(R)}{U_\de(R)-U_\de(\theta R)}= \frac{U(\frac{R}{\de})}{U(\frac{R}{\de})-U(\frac{R \theta}{\de})}.
 \ee
 Choose $\de>0$, small enough so that $\frac{R\theta}{\de}>1$ and thus
  $\frac{U(\frac{R\theta}{\de})}{U(\frac{R}{\de})} \leq \frac{1}{2}$. Therefore, using \eqref{1-12-17-1} we have
   \be \label{EQ6}
 m_{\de,R}= \frac{U(\frac{R}{\de})}{U(\frac{R}{\de})-U(\frac{R\theta}{\de})}\leq \frac{U(\frac{R}{\de})}{U(\frac{R\theta}{\de})}\leq \frac{c_2}{\big(\frac{R}{\de} \big)^{\frac{N-s_1p}{p-1}}}\times \frac{\big(\frac{R \theta}{\de} \big)^{\frac{N-s_1p}{p-1}}}{c_1}=\frac{c_2}{c_1}\theta^{\frac{(N-s_1p)}{p-1}}.
 \ee

 Consider the radially symmetric non-increasing function $\bar{u}_{\de,R}:[0,+\infty) \to \R$ by $$\bar{u}_{\de,R}(r)=G_{\de,R}(U_\de(r)).$$
Then we observe that, $\bar{u}_{\de,R}$ satisfies:
 	\begin{equation}
 	\bar{u}_{\de,R}(r)=\begin{cases}
 	U_\de(r), \quad r \leq R\\
 	0, \quad r \geq \theta R.\\
 	\end{cases}
 	\end{equation}
 	Therefore, we have the following estimates from \cite{MPSY}.
 	
 	\begin{lemma}\label{LEMMA1}\cite[Lemma 2.7]{MPSY}
 	For any $R>0,$	there exists $C=C(N,p,s_1)>0$ such that for any $\de \leq \frac{R}{2},$
 		\begin{equation}\label{G1}
 		\|\bar{u}_{\de,R}\|_{0,s_1,p}^p \leq (S_{s_1,p})^{N/s_1p}+C\big(\frac{\de}{R}\big)^{\frac{N-s_1p}{p-1}},
 		\end{equation}
 		\begin{equation}\label{G2}
 		|\bar{u}_{\de,R}|_p^p \geq \begin{cases}
 		\frac{1}{C}\de^{s_1p} log(R/\de), \quad N=s_1p^2\\
 		\frac{1}{C}\de^{s_1p}, \quad \quad N>s_1p^2
 		\end{cases}
 		\end{equation}
 		and
 		\begin{equation}\label{G3}
 		|\bar{u}_{\de,R}|_{p^*_{s_1}}^{p^*_{s_1}} \geq (S_{s_1,p})^{N/s_1p}-C\big(\frac{\de}{R}\big)^{N/(p-1)}.
 		\end{equation}
 	\end{lemma}
 
 	Let $\eps>0.$ Take $R>0$ be fixed such that $B_{\theta R} \subset \subset \Om.$ Let us define the function $u_{\eps,R}:[0,+\infty) \to \R$ by
\be\lab{24-12-1}u_{\eps,R}(r)=\eps^{-{\frac{(N-s_1p)}{p^2}}}\bar{u}_{\de,R}(r) \quad\mbox{with}\quad \de=\eps^{\frac{(p-1)}{p}}, \quad\forall\, r \geq 0.\ee
 Clearly, $u_{\eps,R} \subset X_{0,s_1,p}(\Om),$ that is, $u_{\eps,R} \equiv 0$ in $\Rn \setminus \Om.$  Therefore, applying \eqref{G1} to \eqref{24-12-1} yields
\be\lab{Est-3}
\|u_{\eps,R}\|_{0,s_1,p}^p \leq (S_{s_1,p})^{N/s_1p}\eps^{-\frac{(N-s_1p)}{p}}+O(1).
\ee
 
 \begin{lemma}\lab{l:27-1}
$|u_{\eps,R}|_{p^*_{s_1}}^{p} = (S_{s_1,p})^{\frac{N-s_1p}{s_1p}}\eps^{-\frac{(N-s_1p)}{p}}+O(1).$
\end{lemma}
 \begin{proof}
  Applying \eqref{G3}, it is easy to see that  $$|u_{\eps,R}|_{p^*_{s_1}}^{p} \geq (S_{s_1,p})^{\frac{N-s_1p}{s_1p}}\eps^{-\frac{(N-s_1p)}{p}}+O(1).$$ To see the upper estimate, we observe that
 \Bea
	|u_{\eps,R}|_{p^*_{s_1}}^{p^*_{s_1}}= \Iom \eps^{-\frac{(N-s_1p)p^*_{s_1}}{p^2}}|\bar{u}_{\de,R}|^{p^*_{s_1}}dx&=& \eps^{-N/p} \Iom |G_{\de,R}(U_\de(x))|^{p^*_{s_1}}dx\\
	&\leq& \eps^{-N/p}|G_{\de, R}'|_{L^{\infty}}^{p^*_{s_1}}\Iom |U_\de(x)|^{p^*_{s_1}}dx\\
	&\leq& \eps^{-N/p} \max\{m^{p^*_{s_1}}_{\de,R},1\}\Iom |U_\de(x)|^{p^*_{s_1}}dx,\\
	\Eea
where in the last line we have used \eqref{EQ3}. Next, applying \eqref{EQ6} to the last line, we have	
		\Bea
		|u_{\eps,R}|_{p^*_{s_1}}^{p^*_{s_1}} \leq C\eps^{-N/p}\int_{\Rn}|U_\de(x)|^{p^*_{s_1}}dx
		&\leq& C\eps^{-N/p}\frac{1}{\de^{\frac{(N-s_1p)p^*_{s_1}}{p}}}\int_{\Rn}|U(\frac{x}{\de})|^{p^*_{s_1}}dx\\
		&=& C\eps^{-N/p}\int_{\Rn}|U(y)|^{p^*_{s_1}} dy\\
		&=&C \eps^{-N/p}|U|_{p^*_{s_1}}^{p^*_{s_1}}\\
		&=&C\eps^{-N/p}(S_{s_1,p})^{N/s_1p},
	\Eea
	where, in the last line we have used \eqref{27-12-9}(ii).
			Hence, 
			we have,
			$$|u_{\eps,R}|_{p^*_{s_1}}^{p} \leq (C (S_{s_1,p})^{N/s_1p} \eps^{-N/p})^{\frac{p}{p^*_{s_1}}}=C(S_{s_1,p})^{\frac{N-s_1p}{s_1p}}\eps^{-\frac{N-s_1p}{p}}.$$ This completes the proof of the lemma.
 \end{proof}

 \begin{lemma}\label{LEMMA2}
 	Let $u_{\eps,R}$ be defined as above. Then the following estimates hold, that is, for \textbf{$t\geq 1,$}
 	\begin{equation}\label{e-1}
 	|u_{\eps,R}|_t^t \geq \begin{cases}
 	k\eps^{\frac{N(p-1)-t(N-s_1p)}{p}}+O(1), &\quad t>\frac{N(p-1)}{N-s_1p}\\
 	k|ln\, \eps|+O(1), &\quad t=\frac{N(p-1)}{N-s_1p}\\
 	O(1),&\quad t<\frac{N(p-1)}{N-s_1p}
 	\end{cases}
 	\end{equation}
 	and 
 	\begin{equation}\label{e-2}
 	\|u_{\eps,R}\|_{0,s_2,t}^t\leq 
 	 O(1),\quad 1\leq t<\frac{N(p-1)}{N-s_1}.
 	\end{equation}
 	In particular, we have
 	\begin{equation}\label{Est-5}
 	|u_{\eps,R}|_p^p \geq \begin{cases}
 	k\eps^{\frac{p^2s_1-N}{p}}+O(1), &\quad N>p^2s_1\\
 	k|ln\, \eps|+O(1), &\quad N=p^2s_1\\
 	O(1), &\quad N<p^2s_1
 	\end{cases}
 	\end{equation} 
 	where $k$ is a positive constant independent of $\eps.$ 
 	\end{lemma}
\begin{proof}
We have,
\Bea
|u_{\eps,R}|_t^t = \Iom |u_{\eps,R}(x)|^t dx=\int_{\Rn}|u_{\eps,R}(x)|^t dx
&\geq& \int_{B_R(0)}|u_{\eps,R}(x)|^tdx\\
&=& \eps^{-\frac{(N-s_1p)t}{p^2}}\int_{B_R(0)} \big(\bar{u}_{\de,R}(x)\big)^t dx\\
&=& \eps^{-\frac{(N-s_1p)t}{p^2}} \int_{B_R(0)}U_\de^t(x)dx\\
&=& \frac{\eps^{-\frac{(N-s_1p)t}{p^2}}}{\de^{\frac{(N-s_1p)t}{p}}}\int_{B_R(0)}U^t(\frac{x}{\de})
dx\\
&=& \eps^{-\frac{(N-s_1p)t}{p^2}}\de^{N-\frac{(N-s_1p)t}{p}}\int_{B_{\frac{R}{\de}}(0)}U^t(x)dx\\
&\geq& \eps^{\frac{N(p-1)}{p}-t\frac{(N-s_1p)}{p}}\int_1^{\frac{R}{\de}}U^t(r)r^{N-1}dr\\
	&\geq& c_1^t\eps^{\frac{N(p-1)}{p}-t\frac{(N-s_1p)}{p}}\int_1^{\frac{R}{\de}}
 \frac{r^{N-1}}{r^{\frac{N-s_1p}{p-1}t}}dr.
\Eea	
	If $t>\frac{N(p-1)}{N-s_1p},$ then we have
	
$$|u_{\eps,R}|_t^t \geq
\frac{c_1^t\eps^{\frac{N(p-1)}{p}-t\frac{(N-s_1p)}{p}}}{\frac{(N-s_1p)t}{p-1}-N}[1-\big(\frac{R}{\de}\big)^{N-\frac{(N-s_1p)t}{p-1}}].$$

 
Since $\de=\eps^\frac{p-1}{p}$, choosing $\eps>0$ small enough we can make $\de$ suitably small so that $1-\big(\frac{R}{\de}\big)^{N-\frac{(N-s_1p)t}{p-1}} \geq \frac{1}{2}.$
Therefore, $$|u_{\eps,R}|_t^t \geq k\eps^{\frac{N(p-1)}{p}-t\frac{(N-s_1p)}{p}},$$
where $k=\frac{c_1^p}{2\big(\frac{t(N-s_1p)}{p-1}-N\big)}.$\\
If $\frac{t(N-s_1p)}{p-1}=N,$ then
$$|u_{\eps,R}|_t^t \geq c_1^t\int_1^{\frac{R}{\de}}\frac{1}{r}dr
= c_1^t (ln\, R-ln\, \eps^{\frac{p-1}{p}})
\geq k |ln\, \eps|+O(1). $$

On the other hand for $\frac{t(N-s_1p)}{(p-1)}<N,$ we have
\Bea
|u_{\eps,R}|_t^t  &\geq& c_1^t \eps^{\frac{N(p-1)}{p}-\frac{t(N-s_1p)}{p}}
\frac{(R/\de)^{N-\frac{t(N-s_1p)}{p-1}}-1}{N-\frac{t(N-s_1p)}{p-1}}\\
&=& c_1^t \bigg[\frac{R^{N-\frac{t(N-s_1p)}{p-1}}-\eps^{\frac{N(p-1)-t(N-s_1p)}{p}}}
	{N-\frac{t(N-s_1p)}{p-1}} \bigg]\\
&\geq& O(1).
\Eea
 To see the proof of (\ref{e-2}), first we note that from Lemma \ref{Lemma4} we have $$u_{\eps,R}\in X_{0,s_1,p}(\Om)\subset X_{0,s_2,t}(\Om), \quad 1\leq t\leq p, \quad 0<s_2<s_1<1$$ and
 $$\|u_{\eps,R}\|_{0,s_2,t}\leq \|u_{\eps,R}\|_{0,s_1,t}.$$
 Therefore,
  \bea\label{EQ}
 \|u_{\eps,R}\|_{0,s_2,t}^t&\leq& \|u_{\eps,R}\|_{0,s_1,t}^t\no\\
 &=& \eps^{-\frac{(N-s_1p)t}{p^2}}\|\bar{u}_{\de,R}(.)\|_{0,s_1,t}^t\no\\
 &=&\eps^{-\frac{(N-s_1p)t}{p^2}}\|G_{\de,R}(U_\de(.))\|_{0,s_1,t}^t\no\\
 &=&\eps^{-\frac{(N-s_1p)t}{p^2}}\displaystyle \int_{\R^{2N}}\frac{|G_{\de,R}(U_\de(x))-G_{\de,R}(U_\de(y))|^t}{|x-y|^{N+s_1t}}dxdy\no\\
 &\leq& \eps^{-\frac{(N-s_1p)t}{p^2}}\displaystyle \int_{\R^{2N}}\frac{|G'_{\de,R}\big(U_\de(x)+\tau (U_\de(y)-U_\de(x)\big)|^t|U_\de(x)-U_\de(y)|^t}{|x-y|^{N+s_1t}}dxdy,\no\\
 \hfill
 \eea
for some $\tau\in(0,1)$.  In the last line, we have used mean value theorem.
Thus from (\ref{EQ3}), we obtain
 \begin{equation}\label{EQ7}
 G'_{\de,R}\big(U_\de(x)+\tau(U_\de(x)-U_\de(y)\big) \leq 1+\frac{c_2}{c_1}\theta^{\frac{N+s_1p}{p-1}}=c_3.
 \end{equation}
Substituting (\ref{EQ7}) into (\ref{EQ}) yields
 \Bea
 \|u_{\eps,R}\|_{0,s_2,t}^t &\leq& \eps^{-\frac{(N-s_1p)t}{p^2}}c_3^t \displaystyle\int_{\R^{2N}}\frac{|U_\de(x)-U_\de(y)|^t}{|x-y|^{N+s_1t}}dxdy\\
 &=&C\eps^{-\frac{(N-s_1p)t}{p^2}}\frac{\de^{N-s_1t}}{\de^{\frac{(N-s_1p)t}{p}}} \displaystyle\int_{\R^{2N}}\frac{|U(z)-U(w)|^t}{|z-w|^{N+s_1t}}dzdw\\
 &=&C\eps^{-\frac{(N-s_1p)t}{p^2}}\eps^{\frac{N(p-t)(p-1)}{p^2}}\|U\|_{0,s_1,t}^t\\
 &=&C\eps^{{\frac{1}{p^2}}\big(N(p-1)(p-t)-(N-s_1p)t\big)}\|U\|_{0,s_1,t}^t,
 \Eea
 where we have used that $\de=\eps^{\frac{p-1}{p}}.$ 
 Note that $t<\frac{N(p-1)}{N-s_1}$ which implies, $$N(p-1)(p-t)-(N-s_1p)t>0.$$
Therefore, we obtain
 \begin{equation}
 \|u_{\eps,R}\|_{0,s_2,t}^t \leq  O(1) \quad\mbox{for}\quad 1\leq t<\frac{N(p-1)}{N-s_1}.
 \end{equation}
 This completes the proof of Lemma \ref{LEMMA2}.

\end{proof}

 \begin{lemma}\lab{lem-2'}
Assume \eqref{1-12-17-3} holds.  Then, for any $\theta>0$, $C_{\theta} \in \big(0,\frac{s}{N}(S_{s_1,p})^{N/s_1p}\big)$, where $C_\theta$ is defined as in \eqref{1-12-17-2}.
 \end{lemma}

 \begin{proof}

As we have fixed $R$,  we take $u_\eps:=u_{\eps,R}.$ Define 
 \begin{equation}\label{v-eps-i}
 v_\eps(x)=\frac{u_\eps(x)}{|u_\eps|_{p^*_{s_1}}}.
 \end{equation}
 Thus $ |v_\eps|_{p^*_{s_1}}=1.$  Define
\Bea
g(t):&=&I_\theta(t v_\eps)\\
&=&\frac{t^p}{p}\|v_\eps\|_{0,s_1,p}^p+\frac{t^q}{q}\|v_\eps\|_{0,s_2,q}^q-\theta\frac{t^r}{r}|v_\eps|_r^r-\frac{t^{p^*_{s_1}}}{p^*_{s_1}}.
\Eea
Since $g$ is a continuous function and $g(0)=0$, $\lim_{t \to +\infty}g(t)=-\infty,$ there exists $t_\eps>0$ such that
$$\sup_{t \geq 0} I_\theta(tv_\eps)=I_\theta(t_\eps v_\eps).$$
Then, $t_\eps$ satisfies $g'(t_\eps)=0$ i.e., 
\begin{equation}\label{k1}
t_\eps^{p-1}\|v_\eps\|_{0,s_1,p}^p+t_\eps^{q-1}\|v_\eps\|_{0,s_2,q}^q-\theta t_\eps^{r-1}|v_\eps|^r_r-t_\eps^{p^*_{s_1}-1}=0.
\end{equation} Consequently,
\begin{equation}\label{k2}
\|v_\eps\|_{0,s_1,p}^p+t_\eps^{q-p}\|v_\eps\|_{0,s_2,q}^q>t_\eps^{p^*_{s_1}-p}.
\end{equation}
As $q<\frac{N(p-1)}{N-s_1}$, combining \eqref{Est-3}, Lemma \ref{l:27-1} and \eqref{e-2}  we have 
\begin{equation}\label{k3}
\|v_\eps\|_{0,s_1,p}^p \leq S_{s_1,p}+O(\eps^{\frac{N-s_1p}{p}}),\quad
\|v_\eps\|_{0,s_2,q}^q \leq \frac{\|u_\eps\|_{0,s_2,q}^q}{|u_\eps|_{p^*_{s_1}}^q}=O(\eps^{\frac{q(N-s_1p)}{p^2}}).
\end{equation}
Therefore, from \eqref{k2} and \eqref{k3}, we see that for any  $\tilde{\eps}>0$  small enough, there exists $t^0_{\tilde{\eps}}>0$ such that for all $\eps \leq \tilde{\eps}$ we have,
$t_{\eps} \leq t^0_{\tilde{\eps}}.$ Using (\ref{k1}) we have,
\begin{equation}\label{k4}
\|v_\eps\|_{0,s_1,p}^p<\theta t_\eps^{r-p}|v_\eps|_r^r+t_\eps^{p^*_{s_1}-p}.
\end{equation}
Using (\ref{k3})-(\ref{k4}) we say there exists $T>0$ such that for any $\eps>0,$ $t_\eps \geq T.$ \\
Let $h(t)=\frac{t^p}{p}\|v_\eps\|_{0,s_1,p}^p-\frac{t^{p^*_{s_1}}}{p^*_{s_1}}.$ 
Then $h(t)$ attains its maximum at $t_0=(\|v_\eps\|_{0,s_1,p}^p)^{\frac{1}{p^*_{s_1}-p}}.$ 
We note that, $N>p^2s_1>ps_1$  implies $N(p-1)<p(N-ps_1)$, Therefore, $\frac{N(p-1)}{N-ps_1}<p<r$. Hence, for $\eps \leq \tilde{\eps}$, applying Lemma \ref{LEMMA2} and Lemma \ref{l:27-1} we obtain,
\Bea
g(t_\eps)&=&h(t_\eps)+\frac{t_\eps^q}{q}\|v_\eps\|_{0,s_2,q}^q-\theta \frac{t_\eps^r}{r}|v_\eps|_r^r\\
&\leq& h(t_0)+\frac{(t^0_{\tilde{\eps}})^q}{q}\|v_\eps\|_{0,s_2,q}^q-\theta \frac{T^r}{r}|v_\eps|_r^r\\
&\leq& \frac{s_1}{N}(S_{s_1,p})^\frac{N}{s_1p}+c_1\eps^{\frac{(N-s_1p)}{p}}+c_2\eps^{\frac{q(N-s_1p)}{p^2}}-c_3\eps^{\frac{(p-1)}{p}(N-\frac{r(N-s_1p)}{p})}, \Eea
with $c_1,c_2,c_3>0$ (independent of $\eps.$) As $$\frac{N-s_1p}{p}>\frac{q(N-s_1p)}{p^2}>\frac{(p-1)}{p}\bigg(N-\frac{r(N-s_1p)}{p}\bigg)>0,$$ choose $\eps>0$ small so that $g(t_\eps)=\sup_{t \geq 0}I_\theta(tv_\eps)<\frac{s_1}{N}(S_{s_1,p})^\frac{N}{s_1p}.$

Hence, $C_\theta \in \big(0,\frac{s_1}{N}(S_{s_1,p})^\frac{N}{s_1p}\big)$ for any $\theta>0.$ 
\end{proof}

\begin{lemma}\label{lem-3}
Assume \eqref{1-12-17-3} holds. Then for any $\theta>0,$  $c_\theta=C_\theta$, where $c_\theta$ and $C_\theta$ are defined as in \eqref{24th Jan} and \eqref{1-12-17-2} respectively.
\end{lemma}
\begin{proof}
Using lemmas \ref{lem-1} and \ref{lem-2}	we conclude that, for any $\theta>0$ there exists $u_\theta \in X_{0,s_1,p}(\Om)$ such that $I_\theta(u_\theta)=C_\theta$ and $I'_\theta(u_\theta)=0. $ Also for any $u \in N_\theta$, we have
\begin{equation}\label{w1}
0=\<I'_\theta(u),u\>= \|u\|_{0,s_1,p}^p+\|u\|_{0,s_2,q}^q-\theta|u^+|_r^r-|u^+|_{p^*_{s_1}}^{p^*_{s_1}}.
\end{equation} 
Therefore, if we define $f(t):= I_\theta(tu)$, where $u\in N_{\theta}$, then  a straight forward computation yields that $f'(1)=0$ and $f''(1)<0$, i.e, 
\begin{equation}\label{w2}
\max_{t \geq 0}I_\theta(tu)=I_\theta(u).
\end{equation}

Observe that, from the definition of $C_\theta$ it follows $C_\theta\leq \max_{t \geq 0}I_\theta(tu) $. Consequently, we obtain $I_\theta(u) \geq C_\theta $ for all $u \in N_\theta.$  Hence, 
\begin{equation}\label{w4}
c_\theta=\inf_{u \in N_\theta}I_\theta(u) \geq C_\theta.
\end{equation}
On the other hand, $u_{\theta} \in N_\theta$ and $I_\theta(u_\theta)=C_\theta$ implies $C_\theta \geq c_\theta$. Hence $c_\theta=C_\theta.$ 
\end{proof}

From the definition of $C_\theta,$ it is easy to see that
$$C_{\theta_1} \leq C_{\theta_2} \quad\mbox{if}\quad \theta_2 \leq \theta_1. $$
Therefore, using Lemma \ref{lem-3}, we also have
$$c_{\theta_1} \leq c_{\theta_2} \quad\mbox{if}\quad \theta_2 \leq \theta_1, $$
which implies $c_\theta$ is non-increasing in $\theta.$ Therefore, for any $\la>0,$ there exists $\rho=\rho(\la)$ (depending on the Mountain Pass Geometry) such that
$0 <\rho \leq c_\theta \leq c_0$ for all $\theta \in [0,\la],$ where $c_0$ is the MP level associated to the functional
$$I_0(u)=\frac{1}{p}\|u\|_{0,s_1,p}^p+\frac{1}{q}\|u\|_{0,s_2,q}^q-\frac{1}{p^*_{s_1}}|u^+|_{p^*_{s_1}}^{p^*_{s_1}}.$$
\begin{lemma}\label{lem-4}
$c_0=\frac{s_1}{N}(S_{s_1,p})^{N/s_1p}.$ 
\end{lemma}
\begin{proof}
	Recall $v_\eps(x)=\frac{u_\eps(x)}{|u_\eps|_{p^*_{s_1}}}$ where $u_\eps=u_{\eps, R}$ is defined as in \eqref{24-12-1}. Arguing as in Lemma \ref{lem-2'}, there exists $t_\eps>0$ such that $\frac{d}{dt}I_0(tv_\eps)|_{t=t_\eps}=0,$ that is,
	\begin{equation}\label{w6}
	t_\eps^{p-1}\|v_\eps\|_{0,s_1,p}^p+t_\eps^{q-1}\|v_\eps\|_{0,s_2,q}^q=t_\eps^{p^*_{s_1}-1}.
	\end{equation}
	Hence, $t_\eps^{p^*_{s_1}-p}\geq\|v_\eps\|_{0,s_1,p}^p.$
	Also, $t_\eps$ is bounded. Using $1<q<p<p^*_{s_1},$ (\ref{w6}) and (\ref{k3}) we have,
	$$t_\eps=\big(S_{s_1,p}+O(\eps^{\frac{q(N-s_1p)}{p^2}}) \big)^{\frac{1}{p^*_{s_1}-p}}.$$
	Therefore,
	\bea\label{w7}
	c_0 \leq I_0(t_\eps v_\eps)&=&\frac{1}{p}\big(S_{s_1,p}+O(\eps^{\frac{q(N-s_1p)}{p^2}})\big)^{\frac{p}{p^*_{s_1}-p}}\big(S_{s_1,p}+O(\eps^{\frac{(N-sp)}{p}})\big)\no\\
	&&+\frac{1}{q} \big(S_{s_1,p}+O(\eps^{\frac{q(N-s_1p)}{p^2}})\big)^{\frac{q}{p^*_{s_1}-p}}O(\eps^{\frac{q(N-s_1p)}{p^2}})\no\\
	&&-\frac{1}{p^*_{s_1}}\big(S_{s_1,p}+O(\eps^{\frac{q(N-s_1p)}{p^2}})\big)^{\frac{p^*_{s_1}}{p^*_{s_1}-p}}\no\\
&=&	\frac{1}{p}\big((S_{s_1,p})^\frac{N-s_1p}{s_1p}+O(\eps^{\frac{q(N-s_1p)}{p^2}})\big)\big(S_{s_1,p}+O(\eps^{\frac{(N-s_1p)}{p}})\big)\no\\
&&+\frac{1}{q} \big((S_{s_1,p})^\frac{q(N-s_1p)}{p^2}+O(\eps^{\frac{q(N-s_1p)}{p^2}})\big)O(\eps^{\frac{q(N-s_1p)}{p^2}})\no\\
&&-\frac{1}{p^*_{s_1}}\big((S_{s_1,p})^\frac{N}{s_1p}+O(\eps^{\frac{q(N-s_1p)}{p^2}})\big)\no\\
&=&\big(\frac{1}{p}-\frac{1}{p^*_{s_1}}\big)(S_{s_1,p})^\frac{N}{s_1p}+O(\eps^{\frac{q(N-s_{s_1}p)}{p^2}})+O(\eps^{\frac{N-s_{s_1}p}{p}})\no\\
	&\to& \frac{s_1}{N}(S_{s_1,p})^{N/s_1p}, \quad\text{as}\quad \eps \to 0.
	\eea
	
	Let $\{u_n\}_{n \geq 1} \subset X_{0,s_1,p}(\Om)$ such that $I_0(u_n) \to c_0$ and $I'_0(u_n) \to 0$ in $(X_{0,s_1,p})'$ as $n \to \infty.$
	Arguing as in Claim 1 of  Lemma \ref{lemma3}, it follows $\{\|u_n\|_{0,s_1,p}\}_{n \geq 1}$ is bounded.  Moreover,  as in \eqref{+8} w.l.g up to a subsequence 
  we can assume  $$\|u_n\|_{0,s_1,p}^p=a+o(1),\quad \|u_n\|_{0,s_2,q}^q=b+o(1),\quad |u_n^+|_{p^*_{s_1}}^{p^*_{s_1}}=a+b+o(1).$$
 Since $2 \leq q<p$, estimating $\<I_0'(u_n),u_n^-\>$ as in the proof of Theorem \ref{thm1}, we obtain $\|u_n^-\|_{0,s_1,p}^p \to 0$ and $ \|u_n^-\|_{0,s_2,q}^q \to 0$ as  $n \to \infty.$ 

 \vspace{2mm}
 
 Therefore, we may assume $u_n \geq 0.$ Hence, $|u_n|_{p^*_{s_1}}^{p^*_{s_1}}=a+b+o(1).$ 
 Set $v_n(x)=\frac{u_n(x)}{|u_n|_{p^*_{s_1}}}.$ Then $|v_n|_{p^*_{s_1}}=1$ and
 $$S_{s_1,p} \leq \|v_n\|_{0,s_1,p}^p=\frac{a+o(1)}{\big(a+b+o(1)\big)^{p/{p^*_{s_1}}}}\leq (a+o(1))^{s_1p/N}.$$
 Hence, we have,
 \bea\label{H4}
 \frac{s_1}{N}(S_{s_1,p})^{N/s_1p}&\leq& \frac{s_1(a+o(1))}{N}\no\\
 &\leq& \frac{s_1(a+o(1))}{N}+\bigg(\frac{1}{q}-\frac{1}{p^*_{s_1}}\bigg)(b+o(1))\no\\
 &\to& c_0,\quad\text{as}\quad n \to \infty.
 \eea
Combining (\ref{w7}) and (\ref{H4}), we have
 $c_0=\frac{s_1}{N}(S_{s_1,p})^{N/s_1p}.$ Hence, proved.
 \end{proof}
 \textit{Remark}:
 \begin{itemize}
 	\item [(i)]
 	For any bounded domain $\Omega \subset \Rn,$ the MP level of the functionals
 	$$I_{0,\Om}(u)=\frac{1}{p}\|u\|_{0,s_1,p}^p+\frac{1}{q}\|u\|_{0,s_2,q}^q-\frac{1}{p^*_{s_1}}|u^+|_{p^*_{s_1}}^{p^*_{s_1}}$$ and
 	$$\tilde{I}_{0,\Om}(u)=\frac{1}{p}\|u\|_{0,s_1,p}^p+\frac{1}{q}\|u\|_{0,s_2,q}^q-\frac{1}{p^*_{s_1}}|u|_{p^*_{s_1}}^{p^*_{s_1}}$$ is $\frac{s_1}{N}(S_{s_1,p})^{\frac{N}{s_1p}},$ so the MP level is independent of $\Om.$ 
 	\item[(ii)]
Using the proof of Lemma \ref{lem-4}, we may assume that all the PS sequence of $I_\theta$ are non-negative. 
 \end{itemize}
	\begin{lemma}\label{lem-5}
		Let $\theta_n \to 0$ as $n \to \infty.$ Then $c_{\theta_n} \to c_0$ as $n \to \infty.$
	\end{lemma}
 \begin{proof}
 	From the definition of $c_\theta,c_0$ we note that
 	\begin{equation}\label{w8}
 	c_{\theta_n} \leq c_0 \quad\forall\quad n\in \N.
 	\end{equation}
 	Let $\{u_n\}_{n \geq 1} \subset X_{0,s_1,p}(\Om)$ such that $u_n \geq 0$ and  satisfies $I_{\theta_n}(u_n)=c_{\theta_n}, I'_{\theta_n}(u_n)=0$ and let $\{t_n\}_{n \geq 1} \subset \R$  such that $t_n u_n \in N_0.$
 	Hence, $c_0 \leq I_0(t_nu_n)=I_{\theta_n}(t_nu_n)+\frac{\theta_n t_n^r}{r}|u_n|_r^r.$ Consequently, 
 	\begin{equation}\label{w9}
 	c_0 \leq c_{\theta_n}+\frac{\theta_nt_n^r}{r}|u_n|_r^r.
 	\end{equation}
 	As $c_{\theta_n} \leq c_0,$ we can show as before $\{\|u_n\|_{0,s_1,p}\}_{n \geq 1}$ is bounded. We also claim that $\{t_n\}_{n \geq 1}$ is bounded. Suppose not. Then up to a subsequence, $t_n \to \infty.$ Note that, $t_nu_n \in N_0$ implies
 	\be\lab{26-12-1}\|u_n\|_{0,s_1,p}^p+t_n^{q-p}\|u_n\|^q_{0,s_2,q}=t_n^{p^*_{s_1}-p}|u_n|_{p^*_{s_1}}^{p^*_{s_1}}.\ee
Since $q<p<p^*_{s_1}$ and $\max\{\|u_n\|_{0,s_2,q}, |u_n|_{p^*_{s_1}}\}   \leq C\|u_n\|_{0,s,p}$ ,  we obtain $\text{RHS of \eqref{26-12-1}}\to\infty$ but LHS remains bounded. Hence the claim follows.

	By the above claim and (\ref{w9}), we have 
 	$$c_0 \leq \liminf_{n \to \infty}c_{\theta_n} \leq \limsup_{n \to \infty}c_{\theta_n} \leq c_0.$$
 	Hence, $c_0=\lim_{n \to \infty} c_{\theta_n}.$ This completes the proof. 
 	 \end{proof}

Since $\Om \subset \Rn$ is a smooth domain, there exists $\delta>0$ such that 
 	$$\Om_\de^+:=\{x \in \Rn \,|\, \text{dist}(x,\Om)<\de \}$$ and
 	$$\Om_\de^-:=\{x \in \Rn \,|\, \text{dist}(x,\Om)>\de \}$$ are homotopically equivalent to $\Om.$ Without loss of generality, we may assume that\\ 
 	$B_\de=B(0,\de) \subset \Om.$ Define, 
 	$$X_{0,s_1,p}^{rad}(B_\de):=\{u\in X_{0,s_1,p}(B_\de) \,|\, u \,\,\mbox{is radial}\}.$$
 	Let $N_{\theta,B_\de}:=\inf\bigg\{u \in X_{0,s_1,p}^{rad}(B_\de)\setminus \{0\}|\<I'_{\theta,B_\de}(u),u\>=0\bigg\}$ where\\
 	$$I_{\theta,B_\de}(u)=\frac{1}{p}\|u\|_{0,s_1,p}^p+\frac{1}{q}\|u\|_{0,s_2,q}^q-\frac{\theta}{r}\int_{B_\de}|u^+|^r\, dx-\frac{1}{p^*_{s_1}}\int_{B_\de}|u^+|^{p^*_{s_1}}\, dx.$$
Denote $n_\theta=\inf_{u \in N_{\theta,B_\de}}I_{\theta,B_\de}(u).$ We note that $n_\theta$ is non-increasing in $\theta.$ Let us denote the MP level for $I_{\theta,B_\de}$ on $X_{0,s,p}^{rad}(B_\de)$ by $\tilde{n_\theta}.$ We also observe that 
$\tilde{n_\theta}>0$ for all $\theta \geq 0.$

\begin{lemma}\label{lem-6}
Assume \eqref{1-12-17-3} holds. Then, for any $\theta>0,$ the following holds:
	\begin{itemize}
		\item [(a)]
		$I_{\theta,B_\de}$ satisfies the $(PS)_c$ condition for all $c \in \big(0,\frac{s_1}{N}(S_{s_1,p})^\frac{N}{s_1p}\big).$ Moreover, 
$$\tilde{n_\theta} \in \big(0,\frac{s_1}{N}(S_{s_1,p})^\frac{N}{s_1p}\big).$$ 
		\item [(b)]
		$n_\theta=\tilde{n_\theta}.$
		\item [(c)]
		$n_\theta \to \frac{s_1}{N}(S_{s_1,p})^\frac{N}{s_1p} $ as $\theta \to 0.$
	\end{itemize}
\end{lemma} 
 \begin{proof}
Applying Brezis-Lieb lemma, it is not difficult to check that $I_{\theta,B_\de}$ in $X_{0,s_1,p}^{rad}(B_\de)$ satisfies the (PS)$_c$ condition for all $c \in \big(0,\frac{s_1}{N}(S_{s_1,p})^\frac{N}{s_1p}\big)$.  By a similar argument as in Lemma \ref{lem-2}, we also obtain $\tilde{n_\theta} \in \big(0,\frac{s_1}{N}(S_{s_1,p})^\frac{N}{s_1p}\big)$. Further, following the same argument as in Lemma \ref{lem-4} and Lemma \ref{lem-5}, it yields  $n_\theta \to \frac{s_1}{N}(S_{s_1,p})^\frac{N}{s_1p}$ and  $\theta \to 0$ respectively. 
 \end{proof}
 Let us define a map $\tau :N_\theta \to \Rn$ by 
 $$\tau(u):=(S_{s_1,p})^{-\frac{N}{s_1p}}\Iom |u^+(x)|^{p^*_{s_1}}x \,\, dx.$$
 Let us denote $I_\theta^{n_\theta} =\{u \in X_{0,s_1,p}(\Om):\, I_\theta \leq n_\theta\}.$
 \begin{lemma}\label{lem-7}
 	There exists $\theta^*>0$ such that for any $\theta \in (0,\theta^*)$ and  $u \in N_\theta \cap I_{\theta}^{n_\theta}$, it holds $\tau(u) \in \Om_\de^+$ . 
 \end{lemma}
 \begin{proof}
 We will prove this by contradiction. Let us suppose $\theta_n \to 0$ and $u_n \in N_{\theta_n} \cap I_{{\theta_n}}^{n_{\theta_n}}$ but $\tau(u_n) \notin \Om_\de^+.$	We observe that 
 \be
 c_{\theta_n} \leq I_{\theta_n}(u_n)= \frac{1}{p}\|u_n\|_{0,s_1,p}^p+\frac{1}{q}\|u_n\|_{0,s_2,q}^q-\frac{\theta_n}{r}|u_n^+|_r^r-\frac{1}{p^*_{s_1}}|u_n^+|_{p^*_{s_1}}^{p^*_{s_1}}\leq n_{\theta_n}\no
 \ee
 and
 $$\|u_n\|_{0,s_1,p}^p+\|u_n\|_{0,s_2,q}^q-\theta_n|u_n^+|_r^r-|u_n^+|_{p^*_{s_1}}^{p^*_{s_1}}=\<I'_\theta(u_n),u_n\>=0.$$
 It can be shown as before that $\|u_n\|_{0,s_1,p}$ is bounded. Therefore, we have,
 \be\label{y1}
 c_{\theta_n} \leq I_{\theta_n}(u_n)= \frac{1}{p}\|u_n\|^p_{0,s_1,p}+\frac{1}{q}\|u_n\|_{0,s_2,q}^q-\frac{1}{p^*_{s_1}}|u_n^+|_{p^*_{s_1}}^{p^*_{s_1}}+o(1)\leq n_{\theta_n}+o(1)
 \ee
 and
 \begin{equation}\label{y2}
 \|u_n\|_{0,s_1,p}^p+\|u_n\|_{0,s_2,q}^q-|u_n^+|_{p^*_{s_1}}^{p^*_{s_1}}=o(1).
 \end{equation}
 Using (\ref{y1}) and (\ref{y2}) we have,
 \be
 \frac{s_1}{N}\|u_n\|_{0,s_1,p}^p \leq \bigg(\frac{1}{p}-\frac{1}{p^*_{s_1}}\bigg)\|u_n\|_{0,s_1,p}^p+\bigg(\frac{1}{q}-\frac{1}{p^*_{s_1}}\bigg)\|u_n\|_{0,s_2,q}^q\leq n_{\theta_n}+o(1)\no.
 \ee
Consequently, applying  Lemma \ref{lem-6}(c) it yields
 \begin{equation}\label{y3}
 \|u_n\|_{0,s_1,p}^p \leq (S_{s_1,p})^\frac{N}{s_1p}+o(1).
 \end{equation}
 From (\ref{y2}), it follows
 \begin{equation}\label{y4}
 \|u_n\|_{0,s_1,p}^p \leq |u_n^+|_{p^*_{s_1}}^{p^*_{s_1}}+o(1).
 \end{equation}
Define $w_n=\frac{u_n}{|u_n^+|_{p^*_{s_1}}}$, which implies  $|w_n^+|_{p^*_{s_1}}=1.$ Using (\ref{y3}) and (\ref{y4}), we obtain
 \be\label{27-12-1}
 S_{s_1,p} \leq \|w_n\|_{0,s_1,p}^p \leq \frac{\|u_n\|_{0,s_1,p}^p}{|u_n^+|_{p^*_{s_1}}^p} \leq \|u_n\|_{0,s_1,p}^{p-\frac{p^2}{p^*_{s_1}}}+o(1)\leq S_{s_1,p}+o(1).
 \ee
 Hence, the function $\tilde{w_n}(x):=w_n^+(x)$ satisfies
 $$|\tilde{w_n}|_{p^*_{s_1}}=1 \quad\text{and}\quad \|\tilde{w_n}\|_{0,s_1,p}^p \to S_{s_1,p}\quad\text{as}\quad n \to \infty.$$
 Using Theorem \ref{Willem-lemma}, there exists a sequence $(y_n,\la_n) \in \Rn \times \R^+$ such that the sequence $v_n$ defined by
 $$v_n(x)=\la_n^{\frac{(N-ps_1)}{p}}\tilde{w_n}(\la_nx+y_n),$$
converges strongly to some $v\in W^{s_1,p}(\Rn)$. Combining \eqref{27-12-1} and (\ref{y4}), we get
 $$S_{s_1,p}|u_n^+|_{p^*_{s_1}}^p+o(1)=\|u_n\|_{0,s_1,p}^p\leq |u_n^+|_{p^*_{s_1}}^{p^*_{s_1}}+o(1).$$
 Hence, \begin{equation}\label{y5}
 |u_n^+|_{p^*_{s_1}}^{p^*_{s_1}} \geq (S_{s_1,p})^\frac{N}{s_1p} +o(1), \quad n \to \infty.
 \end{equation}
Further, from \eqref{27-12-1} and (\ref{y3}) it follows
 $$S_{s_1,p}|u_n^+|_{p^*}^p+o(1)=\|u_n\|_{0,s,p}^p\leq (S_{s_1,p})^\frac{N}{s_1p}+o(1).$$
 Hence, \begin{equation}\label{y6}
 |u_n^+|_{p^*_{s_1}}^{p^*_{s_1}} \leq (S_{s_1,p})^\frac{N}{s_1p}+o(1).
 \end{equation}
 Using (\ref{y5}) and (\ref{y6}) we conclude that,
\be\lab{27-12-2}|u_n^+|_{p^*_{s_1}}^{p^*_{s_1}} \to (S_{s_1,p})^\frac{N}{s_1p} \quad\text{as}\quad n \to \infty.\ee 
Now,
\Bea
\tau(u_n)=(S_{s_1,p})^{-\frac{N}{s_1p}} \Iom |u_n^+(x)|^{p^*_{s_1}}x\,dx
&=& (S_{s_1,p})^{-\frac{N}{s_1p}} |u_n^+|_{p^*_{s_1}}^{p^*_{s_1}}\Iom \tilde{w_n}^{p^*_{s_1}}(x)x\,dx\\
&=& (S_{s_1,p})^{-\frac{N}{s_1p}}|u_n^+|_{p^*_{s_1}}^{p^*_{s_1}}\int_{\Om}
x\la_n^{-N}v_n^{p^*_{s_1}}(\frac{x-y_n}{\la_n})\,dx\\
&=& (S_{s_1,p})^{-\frac{N}{s_1p}}|u_n^+|_{p^*_{s_1}}^{p^*_{s_1}}\int_{\frac{\Om-y_n}{\la_n}} (\la_n z+y_n)v_n^{p^*_{s_1}}(z)\,dz.
\Eea

Applying dominated convergence theorem via \eqref{27-12-2} and Theorem \ref{Willem-lemma} to the last line of the above expression we obtain
$$\tau(u_n)\to y\int_{\Rn} |v|^{p^*_{s_1}}\, dz =y\in\bar\Om,$$
which is a contradiction to the assumption. Hence the lemma follows. 

 \end{proof}
 
Using Lemma \ref{lem-6}, we can find a non-negative radial function $v_\theta \in N_{\theta,B_\de}$ such that $I_\theta(v_\theta)=I_{\theta,B_\de}(v_\theta)=n_\theta.$ Let us define a map $\ga:\Om_\de^- \to I_\theta^{n_\theta}$ by $\gamma(y)=\psi_y$, where $\psi_y$ is defined as follows

\begin{equation}\lab{27-12-4}
\psi_y(x)=\begin{cases}
v_\theta(x-y), \quad\mbox{if}\quad x \in B_\de(y),\\
0,\quad\mbox{otherwise}.
\end{cases}
\end{equation}
Now, for each $y \in \Om_\de^-$ we have, 
\bea\lab{4-1-2}
(\tau \circ \ga)(y)=\tau\circ\psi_y &=& (S_{s_1,p})^{-\frac{N}{s_1p}}\int_{B_{\de}(y)} v_\theta(x-y)^{p^*_{s_1}}x\,dx\no\\
&=& (S_{s_1,p})^{-\frac{N}{s_1p}}\int_{B_\delta(0)} v_\theta(z)^{p^*_{s_1}}(z+y)dz\no\\
&=&y (S_{s_1,p})^{-\frac{N}{s_1p}}\int_{B_\delta(0)} v_\theta(z)^{p^*_{s_1}}dz+ (S_{s_1,p})^{-\frac{N}{s_1p}}\int_{B_\delta(0)} v_\theta(z)^{p^*_{s_1}}z\,dz.\no\\
\hfill
\eea
Further, using the fact that $v_\theta$ is radial, it is easy to check that
\be\lab{4-1-3}\int_{B_\delta(0)} v_\theta(z)^{p^*_{s_1}}z\,dz=0.\ee
Substitution of \eqref{4-1-3} into  \eqref{4-1-2} yields
\be\lab{27-12-7}(\tau \circ \ga)(y)=\al_\theta y,\ee
where, $\al_{\theta}=(S_{s_1,p})^{-\frac{N}{s_1p}}\displaystyle\int_{B_\delta(0)} v_\theta(z)^{p^*_{s_1}}dz$.

 \begin{lemma}\label{lem-8}
 	 $\al_\theta \to 1$ if $\theta \to 0$. 
 \end{lemma}
 \begin{proof}
From Lemma \ref{lem-6}, we observe that 
 \be
 n_\theta= I_{\theta,B_\de}(v_\theta)= \frac{1}{p}\|v_\theta\|_{0,s_1,p}^p+\frac{1}{q}\|v_\theta\|_{0,s_2,q}^q-\frac{\theta}{r}\int_{B_{\de}}|v_\theta|^r-\frac{1}{p^*_{s_1}}\int_{B_{\de}}|v_\theta|^{p^*_{s_1}}\leq \frac{s_1}{N}(S_{s_1,p})^\frac{N}{s_1p}\no
 \ee
 and
 $$\|v_\theta\|_{0,s_1,p}^p+\|v_\theta\|_{0,s_2,q}^q-\theta|v_\theta|_r^r-|v_\theta|_{p^*_{s_1}}^{p^*_{s_1}}=0.$$
 By similar argument as in Lemma \ref{lem-7} we have,
 $|v_\theta|_{p^*_{s_1}}^{p^*_{s_1}} \to (S_{s_1,p})^\frac{N}{s_1p}$ as $\theta \to 0.$ Hence the lemma follows.
 
\end{proof}

Let us define a map $H_\theta:[0,1] \times (N_\theta \cap I_\theta^{n_\theta}) \to \Rn$ by
\be\lab{27-12-5} H_\theta(t,u)=\bigg(t+\frac{1-t}{\al_ \theta}\bigg)\tau(u).\ee

\begin{lemma}\label{lem-9}
	There exists $\theta_*>0$ such that for any $\theta \in (0,\theta_*),$ it holds $$H_\theta([0,1] \times (N_\theta \cap I_\theta^{n_\theta})) \subset \Om_\de^+.$$
\end{lemma}
 \begin{proof}
 We will prove it by method of contradiction. Suppose there exists sequence $\theta_n \to 0$ and $(t_n,u_n) \in [0,1] \times (N_\theta \cap I_\theta^{n_\theta})$	such that
 \begin{equation}\label{o1}
 H_{\theta_n}(t_n,u_n) \notin \Om_\de^+ \,\,\ \forall n \in \N.
 \end{equation}
 	As $t_n \in [0,1]$,  up to a subsequence, we assume $t_n \to t_0 \in [0,1].$ Moreover, by Lemma \ref{lem-8} and from the proof of the Lemma \ref{lem-7},
we have $\al_{\theta_n} \to 1$ and $\tau(u_n) \to y \in \overline{\Om}.$ 	
Hence, $H_{\theta_n}(t_n,u_n)=\big(t_n+\frac{1-t_n}{\al_{\theta_n}}\big)\tau(u_n) \to y \in \overline{\Om}.$ This is a contradiction to (\ref{o1}). Hence the lemma follows.	
\end{proof}

\begin{lemma}\label{lem-11}
	Let $u_\theta$ be a critical point of $I_\theta$ on $N_\theta.$ Then, $u_\theta$ is a critical point of $I_\theta$ on $X_{0,s_1,p}(\Om).$ 
\end{lemma}
\begin{proof}
Suppose, $u_\theta$ is a critical point of $I_\theta$ on $N_\theta.$ Therefore, 
\begin{equation}\label{I'-theta}
\<I'_\theta(u_\theta),u_\theta\>=0.
\end{equation}
Using Lagrange multiplier method, there exists $\mu \in \R$ such that 
\begin{equation}\label{I'-J}
I'_\theta(u_\theta)=\mu J_\theta'(u_\theta),
\end{equation}
where 
\begin{equation}\label{J-theta}
J_\theta(u):=\|u\|_{0,s_1,p}^p+\|u\|_{0,s_2,q}^q-\theta|u^+|_r^r-|u^+|_{p^*_{s_1}}^{p^*_{s_1}}.
\end{equation}
Therefore, \be\lab{27-12-3}\mu\<J'_\theta(u_\theta),u_\theta\>=0.\ee
Observe that,
\bea\label{J'}
\<J'_\theta(u_\theta),u_\theta\>&=&p\|u_\theta\|_{0,s_1,p}^p+q\|u_\theta\|_{0,s_2,q}^q-r\theta|u_{\theta}^+|_r^r-p^*_{s_1}|u_{\theta}^+|_{p^*_{s_1}}^{p^*_{s_1}}\no\\
&=&(p-r)\|u_\theta\|_{0,s_1,p}^p+(q-r)\|u_\theta\|_{0,s_2,q}^q-(p^*_{s_1}-r)|u^+_\theta|_{p^*_{s_1}}^{p^*_{s_1}}<0.
\eea
Consequently, from \eqref{27-12-3} we conclude that $\mu=0$ and therefore by (\ref{I'-J}) we have $I'_\theta(u_\theta)=0$ and this completes the proof.
\end{proof}
In the next two lemmas, we denote $I_{N_\theta} :=I_\theta|_{N_\theta}$(restriction of $I_\theta$ on $N_\theta.$) 
\begin{lemma}\label{lem-12}
Assume \eqref{1-12-17-3} holds and $\theta>0$ is fixed. Then for any sequence $\{u_n\} \subset N_\theta$ such that
	$$I_\theta(u_n) \to c<\frac{s_1}{N}\big(S_{s_1,p}\big)^\frac{N}{s_1p},\quad I'_{N_\theta}(u_n) \to 0,$$
	there exists $u \in N_\theta$ such that up to a subsequence, $u_n \to u$ as $n \to \infty.$
\end{lemma}
\begin{proof} 
From the given assumption, we get there exists a sequence $\{\mu_n\} \subset \R$ such that 
	$$\|I'_\theta(u_n)-\mu_n J'_\theta(u_n)\| \to 0 \quad\mbox{as}\quad n \to \infty.$$
	Hence,
	\begin{equation}\label{J'-theta}
	I'_\theta(u_n)=\mu_n J'_\theta(u_n)+o(1).
	\end{equation}
By (\ref{J'}), we have $\<J'_\theta(u_n),u_n\> <0$ for every $n\geq 1$. Note that, up to a subsequence, $\<J'_\theta(u_n),u_n\> \to l<0$ as $n \to \infty.$ Otherwise, if $\<J'_\theta(u_n),u_n\> \to 0$ as $n \to \infty,$ then $$\|u_n\|_{0,s_1,p} \to 0,\,\, \|u_n\|_{0,s_2,q} \to 0,\,\, |u_n^+|_{p^*_{s_1}} \to 0 \quad\mbox{as}\quad n \to \infty.$$
On the other hand, as $u_n \in N_\theta$ using Sobolev embedding theorem, there exists $C>0$ such that
$$\|u_n\|_{0,s_1,p}^p\leq \|u_n\|_{0,s_1,p}^p+\|u_n\|_{0,s_2,q}^q=\theta|u_n^+|_r^r+|u_n|_{p^*_{s_1}}^{p^*_{s_1}}\leq C(\theta \|u_n\|_{0,s_1,p}^{r}+\|u_n\|_{0,s_1,p}^{p^*_{s_1}}).$$
This in turn implies $$1 \leq C(\theta \|u_n\|_{0,s_1,p}^{r-p}+\|u_n\|_{0,s_1,p}^{p^*_{s_1}-p}),$$ which is a contradiction. Hence, up to a subsequence, we have,
$$\<J'_\theta(u_n),u_n\> \to l<0 \quad\mbox{as}\quad n \to \infty.$$
Moreover, $u_n \in N_\theta $ for all $n \geq 1,$ implies  $\<I'_\theta(u_n),u_n\>=0$ for all $n \geq 1.$ As a consequence, from (\ref{J'-theta}) we have, $\mu_n \to 0$ as $n \to \infty.$ Therefore, $I'_\theta(u_n) \to 0$ as $n \to \infty.$ As $I_\theta(u_n) \to c<\frac{s_1}{N}\big(S_{s_1,p}\big)^\frac{N}{s_1p},$ using Lemma \ref{lem-2} we conclude the result. 
\end{proof}

Define, \be\lab{27-12-8}\theta_{**}=\min\{\theta^*, \theta_*\}, \ee
where $\theta^*$ is same as in Lemma \ref{lem-7} and $\theta_*$ is as found in Lemma \ref{lem-9} . 

\begin{lemma}\label{lem-13}
Assume \eqref{1-12-17-3} holds and $\theta \in (0,\theta_{**})$, where $\theta_{**}$ is as defined in \eqref{27-12-8}. Then
	$$cat_{I_{N_\theta}^{n_\theta}}({I_{N_\theta}^{n_\theta}}) \geq cat_\Om(\Om).$$
\end{lemma}
This follows exactly by the same argument as in \cite[Lemma 4.4]{YinYang1}. For the convenience of the reader, we briefly sketch the proof below.	
\begin{proof}
Let, $cat_{I_{N_\theta}^{n_\theta}}({I_{N_\theta}^{n_\theta}})=n.$ By the definition of $cat_{I_{N_\theta}^{n_\theta}}({I_{N_\theta}^{n_\theta}})$,  we can write $I_{N_\theta}^{n_\theta}=A_1 \cup A_2 \cup \dots \cup A_n$ where $\{A_j\}_{j=1}^n$ are closed and contractible in $I_{N_\theta}^{n_\theta},$ that is,
there exists $h_j \in C([0,1] \times A_j;\, I_{N_\theta}^{n_\theta})$ such that 
$$h_j(0,u)=u, \quad h_j(1,u)=u_0  \quad\forall\quad u \in A_j,$$ where $u_0\in A_j$ is fixed. Let $\ga$ be  as defined in \eqref{27-12-4}. Define, 
$B_j:=\ga^{-1}(A_j),1\leq j \leq n.$ Then, $B_j$ is closed for $1 \leq j \leq n$ and 
$\cup_{j=1}^n B_j=\Om^{-}_{\de}$. Set, $g_j:[0,1] \times B_j \to \Om_\de^+$ by
$$g_j(t,y)=H_\theta(t,h_j(t,\ga(y))), \quad\text{for}\quad \theta \in (0,\theta_{**}),$$
where $H_\theta$ is as defined in \eqref{27-12-5}. Therefore,
$$g_j(0,y)=H_\theta(0,h_j(0,\ga(y)))=\frac{\tau(h_j(0,\ga(y)))}{\al_{\theta}} =\frac{(\tau \circ \ga)(y)}{\al_\theta}=\frac{\al_\theta y}{\al_\theta}=y\quad \forall\, y \in B_j,$$
here we have have used \eqref{27-12-7}.  Further,
$$g_j(1,y)=H_\theta(1,h_j(1,\ga(y)))=\tau(h_j(1,\ga(y)))=\tau(u_0) \in \Om_\de^+,$$  which follows from Lemma \ref{lem-7}.
Therefore, the sets $\{B_j\}_{j=1}^n$ are contractible in $\Om_\de^+.$ Hence, 
$$cat_\Om(\Om)=cat_{\Om_\de^+}(\Om_\de^-) \leq n. $$ This proves the lemma. 
\end{proof}

\vspace{2mm}

{\bf Proof of Theorem \ref{thm2}}: Using Lemma \ref{lem-2} and Lemma \ref{lem-6}, we have for all $\theta>0$,$$c_{\theta},\, n_{\theta}<\frac{s_1}{N}(S_{s_1,p})^\frac{N}{s_1p}.$$
By Lemma \ref{lem-12}, $I_{N_\theta}$ satisfies the (PS)$_c$ condition for all $c \in \big(0,\frac{s_1}{N}(S_{s_1,p})^\frac{N}{s_1p}\big).$ Hence, by Lemma \ref{lem-13}, a standard deformation argument implies that, for $\theta \in (0,\theta_*),\,\ I_{N_\theta}^{n_\theta}$ contains at least $cat_\Om(\Om)$ critical points of the restriction of $I_\theta$ on $N_\theta.$ Now, Lemma \ref{lem-11} implies that $I_\theta$ has at least $cat_{\Om}(\Om)$ critical points on $X_{0,s_1,p}(\Om).$ Now, following the same argument as in Theorem \ref{thm1}, it follows $(P)$ has at least $cat_{\Om}(\Om)$ nontrivial  nonnegative solutions. 

\appendix
\numberwithin{equation}{section}

\appendix
\section{}
Here we first recall the classical deformation lemma from  \cite[Lemma 1.3]{Ambro-Rabin}.

\begin{lemma}\label{l-i}
	Let $J \in C^1(X,\R)$ satisfy (PS)-condition. If $c \in \R$ and $N$ is any neighborhood of $K_c=\{u \in X:\, J(u)=c, \, J'(u)=0\}$, then there exists $\eta(t,x)\equiv \eta_t(x) \in C([0,1] \times X, X)$ and constants $0<\eps<\bar{\eps}$ such that
	\begin{itemize}
		\item [(1)]
		$\eta_0(x)=x$ for all $x \in X.$ 
		\item [(2)]
       $\eta_t(x)=x$ for all $x \in J^{-1}[c-\bar{\eps},c+\bar{\eps}].$
       \item [(3)]
       $\eta_t(x)$ is a homeomorphism of $X$ onto $X$ for all $t \in [0,1].$
       \item [(4)]
       $J(\eta_t(x)) \leq J(x)$ for all $x \in X, t \in [0,1].$
       \item [(5)]
       $\eta_t(A_{c+\eps}-N) \subset A_{c-\eps}$ where $A_c=\{x \in X:\, J(x) \leq c \}$ for any $c \in \R.$ 
       \item [(6)]
       If $K_c=\emptyset,\eta_t(A_{c+\eps}) \subset A_{c-\eps}.$
       \item [(7)]
       If $J$ is even, $\eta_t$ is odd in $x.$ 	
	\end{itemize}
\end{lemma}

Note that the above lemma is also true if $J$ satisfies (PS)$_c$ condition for $c<c_0$ for some $c_0 \in \R.$ Next, recall the general version of Mountain Pass Lemma (see \cite{AI}).
\begin{lemma}\label{lem-10}
	Let $X$ be a Banach space. Let $I \in C^1(X,\R).$ Let us assume for some $\ba,\rho>0,$ we have,
	\item [(i)]
	$I(u)>\ba$ for all $u \in X$ with $\|u\|_{X}=\rho.$
	\item[(ii)] 
	$I(0)=0$ and $I(v_0)<\ba$ for some $v_0 \in X$ with $\|v\|_{X}>\rho.$ \\
	Then there exists a sequence $\{u_n\} \subset X$ such that $I(u_n) \to \al$ and $I'(u_n) \to 0$ in $X'$ as $n \to \infty,$ where $\al$ is given by:
	$$\al:=\inf_{u \in X\setminus \{0\}} \max_{t \geq 0}I(tu).$$ 
\end{lemma}
The next lemma is regarding the elementary properties of Krasnoselskii  genus. 
\begin{lemma}\label{l-ii}
	Let $A,B \in \Sigma.$ Then,
	\begin{itemize}
		\item [(1)]
		if there exists $f \in C(A,B),$ odd, then $\ga(A) \leq \ga(B).$
		\item [(2)]
        if $A \subset B,$ then $\ga(A) \leq \ga(B).$
        \item [(3)]
        if there exists an odd homeomorphism between A and B, then $\ga(A)=\ga(B).$
        \item [(4)]
        if $S^{N-1}$ denotes the unit sphere in $\Rn,$ then $\ga (S^{N-1})=N.$
        \item [(5)]
        $\ga(A \cup B)\leq \ga(A)+\ga(B),$
        \item [(6)]
        If $\ga(A)<\infty,$ then $\ga(\overline{A \cup B}) \geq \ga(A)-\ga(B).$
	    \item [(7)]
	    If $A$ is compact, then $\ga(A)<\infty$ and there exists $\de>0$ such that $\ga(A)=\ga(N_\de(A))$ where $N_\de(A)=\{x \in X:d(x,A) \leq \de \}.$
	    \item [(8)]
	    If $X_0$ is a subspace of $X$ with codimension $k$ and $\ga(A)>k,$ then $A \cap X_0 \neq \emptyset.$ 
	\end{itemize}
\end{lemma}
\begin{proof} See \cite[Lemma 1.2]{Ambro-Rabin} . \end{proof}

\begin{remark}It's easy to observe that if $A$ contains finitely many antipodal points $u_i$, $-u_i$ $u_i\not=0$, then $\ga(A)=1$.
\end{remark}

\vspace{2mm}

{\bf Acknowledgement}  The authors would like to thank Prof. Marco Squassina for pointing out a mistake in the preliminary version of the draft and also for bringing into the notice of authors the paper \cite{MirS} and \cite{MS}. The authors also would like to thank Prof. Giuseppe Mingione for pointing out the connections with the papers \cite{GM-2, GM-1, GM-3}. 
The authors are  thankful to the referee for his/her valuable comments which helped to improve the article in a great extent.The first author is supported by the INSPIRE research grant DST/INSPIRE 04/2013/000152 and the second author is supported by the NBHM grant 2/39(12)/2014/RD-II.

 \end{document}